\documentclass[12pt]{amsart}
\usepackage{amsmath, amsthm, amsfonts, amssymb, color}
 \usepackage{mathrsfs}
\usepackage{amsfonts, amsmath}
 \usepackage{amsmath,amstext,amsthm,amssymb,amsxtra}
 \usepackage{txfonts} 
 \usepackage[colorlinks, citecolor=blue,pagebackref,hypertexnames=false]{hyperref}
 \allowdisplaybreaks
 \usepackage{pgf,tikz}
\begin{document}
\baselineskip 16pt

	\newtheorem{theorem}{Theorem}[section]
	\newtheorem{proposition}[theorem]{Proposition}
	\newtheorem{coro}[theorem]{Corollary}
	\newtheorem{lemma}[theorem]{Lemma}
	\newtheorem{definition}[theorem]{Definition}	
	\newtheorem{assum}{Assumption}[section]
	\newtheorem{example}[theorem]{Example}
	\newtheorem{remark}[theorem]{Remark}
	\newcommand\Kone{K^{j,k}}
    \newcommand\Ktwo{\mathcal{K}^{j,k}}
    \newcommand\KKone{{\mathbb K}^{j,k}}
    \newcommand\KKtwo{\mathcal{B}^{j,k}}
	\newcommand\R{\mathbb{R}}
	\newcommand\RR{\mathbb{R}}
	\newcommand\CC{\mathbb{C}}
	\newcommand\NN{\mathbb{N}}
	\newcommand\ZZ{\mathbb{Z}}
	\def\RN {\mathbb{R}^n}

	\newcommand{\mc}{\mathcal}
	\newcommand\D{\mathcal{D}}	
	\newcommand{\supp}{{\rm supp}{\hspace{.05cm}}}
	\newcommand {\Rn}{{\mathbb{R}^{n}}}
	\newcommand {\rb}{\rangle}
	\newcommand {\lb}{{\langle}}
	\newtheorem{corollary}[theorem]{Corollary}
	\numberwithin{equation}{section}

\title[Singular spherical maximal operators ]
	{Singular spherical maximal operators on a class of degenerate two-step  nilpotent  Lie groups  }

	 \author[ N. Liu and L.X. Yan]{ Naijia Liu, \
 and \ Lixin Yan}

\address{Naijia Liu,
		Department of Mathematics,
		Sun Yat-sen University,
		Guangzhou, 510275,
		P.R.~China}
	\email{liunj@mail2.sysu.edu.cn}	
\address{
Lixin Yan, Department of Mathematics, Sun Yat-sen  University, Guangzhou, 510275, P.R. China}
\email{mcsylx@mail.sysu.edu.cn}

\date{\today}

	\begin{abstract}
		Let $G\cong\mathbb{R}^{d} \ltimes \mathbb{R}$ be a finite-dimensional two-step nilpotent   group   with the group  multiplication
		$(x,u)\cdot(y,v)\rightarrow(x+y,u+v+x^{T}Jy)$ where $J$ is  a  skew-symmetric matrix  satisfying a degeneracy condition
		 with $2\leq {\rm rank}\, J  <d$.
				Consider the maximal function defined by
					$$
{\frak M}f(x, u)=\sup_{t>0}\big|\int_{\Sigma} f(x-ty, u- t x^{T}Jy) d\mu(y)\big|,
$$
where $\Sigma$ is a smooth convex hypersurface and  $d\mu$ is a compactly supported smooth density on $\Sigma$ such that
   the Gaussian curvature of $\Sigma$ is nonvanishing on $\supp d\mu$. In this paper
we prove that when  $d\geq 4$, the maximal operator ${\frak M}$ is bounded on $L^{p}(G)$ for the range $(d-1)/(d-2)<p\leq\infty$.
	\end{abstract}
	
\maketitle

\section{Introduction}
\setcounter{equation}{0}

Let $G$ be a finite-dimensional two-step nilpotent group which we may identify with its Lie algebra ${\frak g}$ by the exponential map.
We assume that $\frak g$ splits as a direct sum ${\frak g}={\frak w}\oplus {\frak z}  $ so that
\begin{align*}
 [{\frak w}, {\frak w}]\subset{\frak z},  \ \ \ \  [{\frak w},   {\frak z}]=\{0\},
\end{align*}
and that $\dim({\frak w})=d,$ $\dim({\frak z})=1$. In exponential coordinates $(x, u)\in \mathbb{R}^{d}\times\mathbb{R}$,
the group multiplication is given by
\begin{align*}
(x, u)\cdot(y, v):=(x+y, u+ v+x^{T}Jy),
\end{align*}
where     $J $ is a
  skew-symmetric matrix  acting on $\mathbb{R}^{d}$ (i.e. $J^T=-J$).
The most prominent examples are the Heisenberg groups ${\mathbb H}^n$
which arise when $d=2n$ and   $J={1\over 2} J_1$ with
\begin{equation}\label{e1.2}
J_1:=
\left(
\begin{array}{cccccc}
0 & -I_n \\
I_n & 0 \\
\end{array}
\right)
\end{equation}
is the standard symplectic matrix on ${\mathbb R}^{2n},$  see \cite{MS}.

 There is a natural dilation structure relative to ${\frak w}$ and ${\frak z}$, namely for $X\in {\frak w}$ and $U\in {\frak z}$
 we consider the dilations
$$
\delta_{t}:(X,U)\rightarrow(tX,t^{2}U).
$$
With the identification of the Lie algebra with the group, $\delta_{t}$ becomes an automorphism of the group.
Let $\Sigma$ be a smooth convex hypersurface in $\frak w$ and let $ \mu$ be a compactly supported smooth density on $\Sigma$.
 The Gaussian curvature of $\Sigma$ does not vanish on the support of $\mu$.  Define the dilate $ \mu_{t}$ by
\begin{equation}\label{e1.3}
\langle   \mu_{t}, f\rangle :=  \int_{\Sigma}  f(tx, 0) d\mu(x).
\end{equation}
Consider  the maximal operator
\begin{equation}\label{e1.4}
{\mathfrak M}f(x, u) :=\sup_{t>0}\big|f\ast  \mu_{t}(x, u) \big|,
\end{equation}
where the convolution is given by
\begin{equation}\label{e1.5}
f\ast  \mu_{t}(x, u):=\int_{\Sigma} f(x-ty, u-tx^{T}Jy)d\mu(y).
\end{equation}

The operator ${\mathfrak M}$ can be understood as an analogue of the classical (Euclidean) spherical maximal
function of Stein \cite{St1, St2} on ${\mathbb R^n}$ (see also Bourgain \cite{Bo}).
In the case of the sphere $\Sigma={\mathbb S}^{2n-1}=\{x: |x|=1\}$ on the noncentral part of
  the Heisenberg group ${\mathbb H}^n$
  and $\mu$ denotes the normalised surface measure on the sphere
 ${\mathbb S}^{2n-1}$, the maximal function \eqref{e1.4} was  studied
by Nevo and Thangavelu in \cite{NT},   and they obtained $L^{p}({\mathbb H}^n)$
boundedness for  $d=2n \geq 4$ and  $p>(d-1)/(d-2)$ by using spectral methods. Later, improving the results in \cite{NT},
   Narayanan and  Thangavelu \cite{NT1} obtained the optimal range   $p>d/(d-1)$   by modifying the argument in
   \cite{NT} and combining it with Stein's square function method.
   In \cite{MS}, M\"uller and Seeger  independently  proved this optimal range for $p$
   by using Fourier integral operators
   in a study which concerns more general surfaces $\Sigma$ with
   a nonvanishing rotational curvature  in ``non-degenerate" two-step nilpotent
 groups including all ${\mathbb H}$-type groups.  This topic has
 attracted a lot of attention in the last decades, and has been a very active research topic
 in harmonic analysis; see, for example  \cite{ACPS,  C, Fi, MS, MSS, NT1, PS1, Sc,  SS} and the references therein
 for further details.

 Throughout this paper, we aasume that  $G$ satisifes the degenerate hypothesis, i.e.,  for every nonzero linear
 functional $\omega\in {\frak z}^{\ast}$, the rank $k$ of the bilinear form
$\mathcal{J}_{\omega}:(X,Y)\rightarrow\omega\big([X,Y]\big)$, which maps ${\frak w}\times {\frak w}$ to $\mathbb{R}$, satisfies $0<k<d$.
Note in exponential coordinates the degenerate hypothesis is equivalent with
$$0< k=\text{rank}J<d,
$$
and   $d$ may be odd.
The skew symmetry of ${\mathcal J_{\omega}}$   implies that $\text{rank} J$ is even,
and so we can assume  that  $2\leq k<d$. In this paper,
we   prove the following result.

\begin{theorem}\label{th1.1}
Suppose $d\geq4$. Then the maximal operator ${\mathfrak M}$ extends to a bounded operator on $L^{p}(G)$ if $p>(d-1)/(d-2)$,
and there exists a constant $C>0$ such that
$$
\|{\mathfrak M}f \|_{L^p(G)}\leq C \|f\|_{L^p(G)}, \ \ \  1<p<\infty.
$$
As a consequence,
$
\lim\limits_{t\to 0} f\ast  \mu_{t}(x, u) =c f(x, u)
$
a.e.
for all $f\in L^p(G)$,  with $c=\int d\mu.$
\end{theorem}

We remark that $p>d/(d-1)$ is necessary in Theorem~\ref{th1.1} and it can be seen by testing ${\frak M}$ on the function
		$f$ given by $f(y,v)= |y|^{1-d}(\log|y|)^{-1}\chi(y,v)$  with a suitable nonnegative
 cutoff function $\chi$, which equals to $1$ if $|(y,v)|\leq \frac{1}{2}$ and $\supp \chi\subseteq B(0,1)$ (see \cite[Chapter XI]{St2} and Section 2 below).
 However, it is not clear for us to establish  $L^p$ boundedness   for range
 $d/(d-1)<p\leq (d-1)/(d-2)$ for $ d\geq 4$.

This article is organized as follows. In Section $2$, we introduce some notations and preliminary lemmas.
We reduce our main theorem to Theorem \ref{th2.1} below. In Section $3$, we prove a useful lemma, which plays an important
role in the proof of Theorem~\ref{th2.1}. The proof of Theorem~\ref{th2.1} will be given in Sections 4 and 5
by reducing the estimates for the averages to estimates for oscillatory integral operators.

Throughout, the letters ``$c$"  and ``$C$" will denote (possibly
different) constants  that are independent of the essential
variables.


\bigskip

\section{Preliminaries}
\setcounter{equation}{0}

The Fourier transform is defined for $f\in {\mathscr S(\mathbb R^n)}$ by
$$
{\hat f} (\xi)= \mathcal{F}(f)(\xi)= (2\pi)^{-n/2} \int_{\mathbb R^n} e^{-ix\xi} f(x)dx
$$
and this definition is then extended to $L^2(\mathbb R^n)$ and to bounded measure in the usual way.
The inverse Fourier transform  is denote by $\mathcal{F}^{-1}(f)={\check f}$.

We recall the definition of convolution
\begin{align*}
f\ast g(x,u) &=\int_{\mathbb R^{d+1}} f(y,v)g((y,v)^{-1}(x,u))dydv\\
&=\int_{\mathbb R^{d+1}} f(x-y,u-v-x^{T}Jy)g(y,v)dydv,
\end{align*}
where $J$ is skew-symmetric $d\times d$ matrix acting on $\mathbb{R}^{d}$ with $2\leq\text{rank} J<d$ by our degenerate hypothesis.
By a rotation we can suppose
\begin{equation}\label{e2.1}
J=
\left(
\begin{array}{cccccc}
J_{k} & 0 \\
0 & 0 \\
\end{array}
\right)
\end{equation}
for a skew-symmetric $k\times k$ matrix $J_{k}$ satisfies $\text{rank} J_{k}=k\in[2,d)$.

Recall that   $\Sigma$ is  a smooth convex hypersurface in $\frak w$ and let $\mu$ be a compactly supported smooth density on $\Sigma$.
 The Gaussian curvature of $\Sigma$ does not vanish on the support of $\mu$.
In the following,  we use notation
$$x=(x_{1}, x'), \ \  \ \ \ \ \ x'= (x_{2},...x_{d})\in \mathbb{R}^{d-1};
$$
 and also
 $$x=(x'',x_{d}), \ \  \ \ \ \ \   x''= (x_{1},...x_{d-1})\in \mathbb{R}^{d-1}.
 $$
By localizations,  we only consider that the projection of $\Sigma$ to $\frak w$ are given in
the following two cases since    other cases  can be studied similarly:

\begin{itemize}
\item[  (1)]   $x_{1}=\Gamma_1(x'),  x'= (x_{2},...x_{d})\ $ with  $\Gamma_1\in C^{\infty}(\mathbb R^{d-1})$
so that $\mu$ is supported in a small  neighborhood
of $(\Gamma_1(x'_{0}), x'_{0})$ for some $x'_{0}\in \R^{d-1}$;

\item[(2)]  $x_{d}=\Gamma_2(x''),  x''= (x_{1},...x_{d-1})\ $  with  $\Gamma_2\in C^{\infty}(\mathbb R^{d-1})$
so that  $\mu$ is supported in a small neighborhood of $(x_{0}'', \Gamma_2(x_{0}''))$ for some $x''_{0}\in \R^{d-1}$.
\end{itemize}

\smallskip

Using the Fourier inversion formula for Dirac measures we may write
\begin{eqnarray}\label{e2.2}
\mu_1(x, u)=\chi_{\mu_1}(x,u)\iint_{\mathbb R \times \mathbb R} e^{i(\sigma(x_1- \Gamma_1(x')) +\tau u)} d\sigma d\tau
\end{eqnarray}
and
\begin{eqnarray}\label{e2.3}
\mu_2(x, u)=\chi_{\mu_2}(x,u)\iint_{\mathbb R \times \mathbb R} e^{i(\sigma(x_d- \Gamma_2(x'')) +\tau u)} d\sigma d\tau
\end{eqnarray}
in the above cases (1) and (2), respectively. Here, $\chi_{\mu_i} ( i=1, 2) $  is a smooth compactly supported function and the integral convergence
 in the sense of oscillary integrals (thus in the sense of distributions).
Then Theorem~\ref{th1.1} follows from the following result:

\begin{theorem}\label{th2.1} Let $\mu_1$ and $\mu_2$ be given as above. Then we we have
\begin{itemize}
\item[(i)]  Suppose $d\geq 4$ and $p>(d-1)/(d-2)$. We have
\begin{align}\label{e2.5}
 \big\|\sup_{t>0}|f\ast ( \mu_1)_{t} |\big\|_{L^p(\mathbb R^{d+1})} \leq C \|f\|_{L^p(\mathbb R^{d+1})}.
\end{align}

\item[(ii)]   Suppose $d\geq 3$.  We have
\begin{align}\label{e2.4}
 \big\|\sup_{t>0}|f\ast ( \mu_2)_{t}  |\big\|_{L^p(\mathbb R^{d+1})} \leq C \|f\|_{L^p(\mathbb R^{d+1})}
\end{align}
holds  if and only if and $p> d/(d-1)$.
\end{itemize}
\end{theorem}

The proof of (i) and (ii) of Theorem~\ref{th2.1} will be given in Sections 4 and 5 below, respectively.
As pointed out in the introduction,  $p> d/(d-1)$ is necessary  in the above theorem
and it can be seen by testing ${\frak M}$ on the function
		$f$ given by $f(y,v)=|y|^{1-d}(\log|y|)^{-1}\chi(y,v)$ with a suitable
 cutoff function $\chi$, which equals to $1$ if $|(y,v)|\leq \frac{1}{2}$ and $\supp \chi\subseteq B(0,1)$.  Let $\mu$ be the induced Lebesgue measure on the sphere $S^{d-1}$.  Then we have
\begin{align*}
|f\ast \mu_{t}(x,u)|
&\geq-\int_{|y|=1,|x-ty|\leq\frac{|x|}{10}}|x-ty|^{1-d}(\log|x-ty|)^{-1}d\mu(y)\\
&\simeq -\int_{0}^{\frac{|x|}{10}}r^{1-d}(\log r)^{-1}dr=\infty
\end{align*}
if $|(x,u)|\leq  \big(10(\|J\|+1)\big)^{-1} $ and $t=|x|$, $d\geq4$. So ${\frak M}f=\sup_{t>0}|f\ast \sigma_{t}|\notin L^{p}(\R^{d+1})$
 for all $1\leq p<\infty$. However,
  $f\in L^{p}(\R^{d+1})$ for $1\leq  p \leq d/(d-1)$, so  ${\frak M}$ is unbounded on $L^p(\R^{d+1})$ for $1\leq p\leq d/(d-1)$.

In the end of this section, we state the following two lemmas, which will be useful in the proof of
Theorem~\ref{th2.1}.

\begin{lemma}\label{le2.2}
Suppose that
$$
\sup_{s\in [1,2]}\bigg(\sum_{n\in {\mathbb Z}} \big\|F_n(\cdot, s)\big\|_2^2\bigg)^{1/2}\leq A_1
$$
and
$$
\sup_{s\in [1,2]}\bigg(\sum_{n\in {\mathbb Z}} \left\|\frac{d}{ds}F_{n}(\cdot,s)\right\|_{2}^{2}\bigg)^{1/2}\leq A_2.
$$
Then
$$
\big\|\sup_n \sup_{s\in [1,2]} |F_n(\cdot, s)|\big\|_2\leq C(A_1 +\sqrt{A_1A_2}).
$$
 \end{lemma}

\begin{proof}
For the proof, see \cite[Lemma 3.1]{MS}, \cite[p. 499]{St2}.
\end{proof}

 We now state an almost orthogonality lemma, the Cotlar-Stein Lemma.

\begin{lemma}\label{le2.3}
Suppose $0<\epsilon<1, A\leq B/2$ and let $\{T_n\}_{n=1}^{\infty}$ be a sequence of bounded operators on a
Hilbert space $H$ so that the operator norms satisfy
$$
\|T_n\|\leq A
$$
and
$$
\|T_n T^{\ast}_{n'}\|\leq B^22^{-\epsilon|n-n'|}.
$$
Then for all $f\in H$
$$
\left(\sum_{n=1}^{\infty} \big\|T_n f\big\|^2\right)^{1/2}\leq CA \sqrt{\epsilon^{-1} {\rm log} (B/A)} \ \|f\|.
$$
\end{lemma}

\begin{proof}
For the proof, see \cite[Lemma 3.2]{MS}.
\end{proof}


\bigskip

 \section{A useful  lemma}
\setcounter{equation}{0}

Let $m, n\in[1,d]$. For $m\times n$ matrix $A$, we define its norm $\|A\|$ as
$$
\|A\|=\sup_{X\in \Rn,|X|=1}|AX|,
$$
where $X=(x_{1},...,x_{n})$ and $|X|=\big(\sum_{i=1}^{n} x_{i}^{2}\big)^{1/2}$.
We define $P_{n}$, $E_{n}$ as the $d\times n$ matrices:
\begin{equation}\label{e3.1}
P_{n}=
\left(
\begin{array}{cccccc}
0 \\
I_{n} \\
\end{array}
\right)\ \text{and} \
E_{n}=
\left(
\begin{array}{cccccc}
I_{n} \\
0 \\
\end{array}
\right)
\end{equation}
and $P_{n}^{T}$, $E_{n}^{T}$ as their transposes.

 Recall that $J$ is skew-symmetric matrix acting on $\mathbb{R}^{d}$   with $2\leq\text{rank} J<d$.
By a rotation we can suppose $J$ has the form \eqref{e2.1}.

In this section, we first prove the following lemma, which plays an essential role in the proof
 of    (i) of Theorem~\ref{th2.1}.

\begin{lemma}\label{le3.1}
\begin{itemize}
\item[(i)]
Let $J_{k}$ be a $k\times k$ skew-symmetric matrix with rank$J_{k}= k\in [2, d)$. For every
  $k\times k$ semi-positive definite matrix $G$,  there exists a constant $c_0=c_0(J_k, G, M)>0$ such that
\begin{equation}\label{e3.3}
\inf_{ \substack{  |a|\leq M \\ X\in \mathbb{R}^{k}, \, |X|=1}} \left\|(aG+J_{k})X \right\|\geq c_0;
\end{equation}
\item[(ii)]
For every $(d-1)\times (d-1)$ positive definite matrix $Q$, there exists a constant $c_1=c_1(J, Q,\delta)>0$ such that
\begin{equation}\label{e3.4}
\inf_{\substack{ |a|\geq \delta\\ Y\in \mathbb{R}^{d}, \, |Y|=1}} \left\|(P_{d-1}QP_{d-1}^{T}+aJ)Y \right\|\geq c_1.
\end{equation}
\end{itemize}
\end{lemma}
\begin{proof}
Note that $G$ is a $k\times k$ semi-positive definite matrix, we can write $G=B^TB$ for some matrix $B$. First, let us show that
\begin{align}\label{e3.5}
\inf_{|a|\geq \delta,\, |X|=1}  \|(G+aJ_{k})X\|\geq c_2:=\min\Big\{ { \delta\| J_k^{-1}\|^{-1}\over 2 },
 {\delta^{2}\| J_k^{-1}\|^{-2} \over 4\|B^{T}\|^{2} } \Big\}.
\end{align}
We argue  \eqref{e3.5} by contradiction. Assume that
there exist some   $|a|\geq \delta$ and $|X|=1$ such that
$
\|(G+aJ_{k})X\|< c_2.
$
Since $J_{k}$  is a skew-symmetric matrix, we have
  $X^{T}J_{k}X =0$. Hence,
\begin{eqnarray*}
\|B X\|^2 = |X^{T}GX|  =  |X^T (G  + aJ_{k})X|  \leq \| (G + aJ_{k} ) X)\|<c_{2}.
\end{eqnarray*}
From this, we see that $\|GX\|\leq \|B^T\| \, \|B X\|<\|B^{T} \| \sqrt{c_{2}}$. By   assumption,
it tells us that
$$
\|J_{k}X\|\leq|a|^{-1} \big(\|(G+aJ_{k})X\|+\|GX\|\big)<\delta^{-1} \big( c_{2}+ \|B^{T} \|\sqrt{c_{2}} \big)\leq\| J_k^{-1}\|^{-1}.
$$
However, $\det J_{k}\neq0$ and  $
1=|X|= \|J_k^{-1} J_kX \|\leq  \| J_k^{-1}\|\, \| J_kX  \|
 $ and so $\| J_k X \|\geq   \| J_k^{-1}\|^{-1}.$
This will be our desired contradiction.

We now apply \eqref{e3.5} to prove \eqref{e3.3}.
 By \eqref{e3.5},
\begin{equation*}
\|(aG+J_{k})X\|=|a| \ \|(G+ a^{-1}J_{k})X\| \geq  |a|  c_2 \geq {c_2\| J_k^{-1}\|^{-1}\over c_2+ \|G\|}.
\end{equation*}
whenever $ \| J_k^{-1}\|^{-1}/ (c_2+ \|G\|) \leq |a|\leq M$. On the other hand,
we use estimate  $\|aGX\|\leq |a|\, \|G\| $ to get that for
  $|a|\leq \| J_k^{-1}\|^{-1}/ (c_2+ \|G\|) $,
 $$\|(aG+J_{k})X\|\geq \|J_{k}X\|- \|aGX\|\geq  \| J_k^{-1}\|^{-1}-  |a|\, \|G\| \geq {c_2\| J_k^{-1}\|^{-1}\over c_2+ \|G\|},
   $$
thereby concluding the proof of \eqref{e3.3} with $c_0:=  c_2\| J_k^{-1}\|^{-1}/ (c_2+ \|G\|)$.

Next let us  show \eqref{e3.4}.
Note that $I_d=P_{d-k}P_{d-k}^{T} + E_kE_k^{T}$.
Then for all $|a|\geq \delta$ and $|Y|=1$,
\begin{align}\label{e3.6}
\|(P_{d-1}QP_{d-1}^{T}+aJ)Y\|&\geq \|(E^{T}_{k} P_{d-1}QP_{d-1}^{T} E_k +aJ_{k}) E^T_kY +E^{T}_{k}P_{d-1}QP_{d-1}^{T}P_{d-k}P_{d-k}^{T}Y \|\nonumber\\
&\geq \|(E^{T}_{k} P_{d-1}QP_{d-1}^{T}E_k +aJ_{k})E^{T}_{k}Y\|
-\|E^{T}_{k}P_{d-1}QP_{d-1}^{T}P_{d-k}P_{d-k}^{T}Y\|\nonumber\\
&=\|(Q_{k}+aJ_{k} )E^{T}_{k}Y\|- c_3 \|P_{d-k}^{T}Y\|,
\end{align}
where   $Q_{k}=E^{T}_{k}P_{d-1}QP_{d-1}^{T}E_{k}$ is a semi-positive definite matrix and $c_3:=\|E^{T}_{k}P_{d-1}QP_{d-1}^{T}P_{d-k}\|.$
 Since $|Y|=1$,    we have $\|E^{T}_{k}Y\| \geq 1-\|P_{d-k}^{T}Y\| $. We apply  \eqref{e3.5}  to know that
 there exists a constant $c_2=c_2(J, Q)>0$ such that
$$
\|(Q_{k}+aJ_{k})E^{T}_{k}Y\|    \geq c_2 \|E^{T}_{k}Y\| \geq c_2(1-\|P_{d-k}^{T}Y\|),
$$
which, in combination with \eqref{e3.6}, shows that for all $|a|\geq \delta$ and $|Y|=1$,
\begin{align}\label{e3.7}
\|(P_{d-1}QP_{d-1}^{T}+aJ)Y \|\geq  c_2 -(c_2+c_3)\|P_{d-k}^{T}Y\|
\geq {\frac{c_2 }{4}}
\end{align}
whenever
$$\|P_{d-k}^{T}Y\|\leq \frac{c_2}{4c_2 +4c_3}.
$$

Consider the case $\|P_{d-k}^{T}Y\|\geq \frac{c_2}{4c_2 +4c_3}$.
Suppose
$Q=A^{T}A$ for some matrix A with $\det A\neq0$.
Since $J $  is a skew-symmetric matrix, we have
  $Y^{T}J Y =0$.  Hence for all $a\in \R$,
\begin{align}\label{e3.8}
\|(P_{d-1}QP_{d-1}^{T}+aJ)Y\|&\geq |Y^T(P_{d-1}QP_{d-1}^{T}+aJ)Y |
=| Y^T P _{d-1}QP_{d-1}^{T}Y |\nonumber\\
& =\|AP_{d-1}^{T}Y\|^{2}\geq \|A^{-1}\|^{-2} \|P_{d-k}^{T}Y\|^2\geq\frac{c_2^{2}\, \|A^{-1}\|^{-2}}{(4c_2 +4c_3)^{2}}
\end{align}
since $\det A\neq 0$.
From \eqref{e3.7} and \eqref{e3.8}, we complete the proof of  (ii) with
$$
c_1:=\min\{\frac{c_{2}}{4},\frac{c_2^{2}\, \|A^{-1}\|^{-2}}{(4c_2 +4c_3)^{2}}\}.
$$
The proof is completed.
\end{proof}


\bigskip

 \section{Proof of (i) of Theorem~\ref{th2.1}}
\setcounter{equation}{0}

In this  section, we will show that (i) of Theorem~\ref{th2.1}. That is,   for  $d\geq 4$,
the maximal operator $\sup_{t>0}|f\ast ( \mu_1)_{t} (x,u)|$ is bounded on $L^p(\mathbb R^{d+1})$
where $\mu_1$ is given as in \eqref{e2.2}, i.e.,
 \begin{eqnarray*}
\mu_1(x, u)=\chi_{\mu_1}(x,u)\iint_{\mathbb R \times \mathbb R} e^{i(\sigma(x_1- \Gamma_1(x')) +\tau u)} d\sigma d\tau
\end{eqnarray*}
with  $ x'= (x_{2},...x_{d})$, $\Gamma_1\in C^{\infty}(\mathbb R^{d-1})$
and   $\mu_1$ is supported in a small  neighborhood
of $(\Gamma_1(x'_{0}), x'_{0})$ for some $x'_{0}\in \R^{d-1}$.

 Let $\varphi\in C_0^{\infty}(\R)$ be an even smooth  function such that $\varphi=1$
 if $|s|\leq 1$ and such that the support of $\varphi$ is contained in $(-2, 2)$. Let $\phi(s)=\varphi(s)-\varphi(2s)$
 and let, for $k\geq 1$, $\phi_k(s):=\phi(2^{-k}s)$. Let $\phi_0(s):=\varphi(s)$.
 Then
 $$\supp \phi_0 \subseteq \{s:   |s|\leq 2 \}, \ \ \ {\rm and} \ \  \ \supp \phi_k \subseteq \{s: 2^{k-1}\leq |s|\leq 2^{k+1} \}
 $$
 for $k\geq 1$,
 and
 we have
\begin{eqnarray} \label{e4.2}
1= \sum\limits_{k=0}^{\infty} \phi_k(s) , \ \ \ \ \forall s\in\R.
\end{eqnarray}

Since $\mu_1$ has a sufficiently small support,  	
 we can choose a smooth nonnegative function $\chi_{1}$
  defined on $\mathbb{R}^{d}$,
which is equal to $1$ on $\supp \mu_1$ and also has a small support.  By the curvature hypothesis and the convexity of $\Sigma$,
we can further assume $\partial_{x'}^{2}\Gamma_1(x')$ is positive
definite for all $x\in\supp \chi_{1}$  such that
\begin{eqnarray} \label{e4.3}
\inf_{x\in \supp\chi_{1}}\left| {\rm det} \left(\partial_{x'}^{2}\Gamma_1(x') \right) \right|\geq c_{H}^{(1)}
\end{eqnarray}
for some $c_{H}^{(1)}>0$.   Define for $j, k\geq 0$,
\begin{align}\label{e4.4}
	K^{j,k}(x,u)&=\chi_{1}(x)\iint_{\mathbb R \times \mathbb R}
	e^{i\sigma(x_{1}-\Gamma_1(x'))}e^{i\tau u}\phi_{j}(\sigma)\phi_{k}(\tau)d\sigma d\tau\nonumber \\
	&=\chi_{1}(x)\mathcal{F}^{-1} (\phi_{j})\big(x_{1}-\Gamma_1(x')\big)\mathcal{F}^{-1}(\phi_{k})\big(u\big)
\end{align}
such that
  \begin{align*}
	\mu_1(x,u)  = \chi_{1}(x)\mu_1(x,u)=\chi_{\mu_1}(x,u)\sum_{j=0}^{\infty} \sum_{k=0}^{\infty}K^{j,k}(x,u).
\end{align*}
For $t>0$, define the dilates
\begin{align}\label{e4.5}
K^{j,k}_t(x,u)=t^{-(d+2)} K^{j,k}(t^{-1}x, t^{-2}u).
\end{align}
These, together with \eqref{e2.2}, tell us that
\begin{align}\label{e4.6}
|f\ast (\mu_1)_{t}|\leq C|f|\ast \left(\chi_{1}(x)\delta(x_{1}-\Gamma_1(x'))\delta(u)\right)_{t}= C\sum_{j=0}^{\infty} \sum_{k=0}^{\infty}  |f|\ast K^{j,k}_{t}.
\end{align}

For  $j=k=0$,    it can be verified   that
$
|K^{0,0}(x,u)|
 \leq C_N (1+|x|+|u|)^{-N}
$
for any $N>0$, and so  $\sup_{t>0}|f\ast K^{0,0}_{t}| $
  is controlled by the appropriate variant of the Hardy-Littlewood maximal function and therefore  (see \cite{St2})  we have the inequality
\begin{align}\label{e4.7}
\big|\big|\sup_{t>0}|f\ast K^{0,0}_{t}|\big|\big|_{L^{p}(\mathbb{R}^{d+1})}
\leq C||f||_{L^{p}(\mathbb{R}^{d+1})}, \ \ \ \  1<p<\infty.
\end{align}

Consider   $j+k\geq 1$.   By definition of $\Kone$, we see that
if $k=0$ and $j\geq1$, then it follows by an integration by parts,
\begin{align}\label{e4.8}
\big|\int_{\mathbb R^{d+1}} K^{j,0}(x,u)dxdu\big|\leq C_{N}2^{-Nj}.
\end{align}
If $k>0$, then
\begin{align*}
\int_{\mathbb R^{d+1}} \Kone(x,u)dxdu&=\int_{\mathbb R^{d+1}} \chi_{1}(x)\mathcal{F}(\phi_{j})\big(x_{1}-\Gamma_1(x')\big)
\mathcal{F}(\phi_{k})(u)dxdu=0.
\end{align*}
Choose a nonnegative function $b\in C^{\infty}_{c}(\mathbb{R}^{d+1})$ such that $\int_{\mathbb R^{d+1}} b(x,u)dxdu=1$. Define
\begin{align}\label{e4.9}
\KKone(x,u)=\Kone(x,u)-\gamma_{j,k}b(x,u).
\end{align}
where
\begin{eqnarray} \label{gamma}
\gamma_{j,k}=
\left\{
\begin{array}{lll}
\int_{\mathbb R^{d+1}} \Kone(x,u)dxdu, & k=0\\[6pt]
0,& k>0,
\end{array}
\right.
\end{eqnarray}
 and by \eqref{e4.8}, $|\gamma_{j,k}|\leq C_{N}2^{-Nj}$ for $k=0$, and  $ \gamma_{j,k}=0$ for $k>0$. Since the maximal operator
 generated by the kernel $b$ is bounded by the nonisotropic  Hardy-Littlewood maximal operator, we see that for $1<p\leq \infty,$
\begin{align} \label{e4.10}
|\gamma_{j,0}|\big\|\sup_{t>0}|f\ast b_{t}|\big\|_{L^{p}(\mathbb{R}^{d+1})}
\leq  C_{N}2^{-Nj} \|f\|_{L^{p}(\mathbb{R}^{d+1})}.
\end{align}
Finally, we have that
\begin{align}\label{e4.11}
\int_{\mathbb R^{d+1}} \KKone(x,u)dxdu=0.
\end{align}

\subsection{Square functions and almost orthogonality }

\begin{lemma}\label{le4.1}
For all $j+k\geq1$, $\theta\in[0,1]$, we have
\begin{align*}
\big\|\sup_{t>0}|f\ast \Kone_{t}|\big\|_{L^{2}(\mathbb{R}^{d+1})}
\leq C(j+k)2^{j/2}C_{j,k}(\theta)||f||_{L^{2}(\mathbb{R}^{d+1})},
\end{align*}
where
\begin{equation}\label{e4.12}
C_{j,k}(\theta):=\begin{cases} 2^{-\theta k} 2^{-(2-d\theta)(d-2)j/4}, & \textrm{if\ } \ M2^{k}\geq 2^{j},\\[4pt]
2^{-(d-1)j/2}, & \ \textrm{if\ } \ M2^{k}\leq 2^{j}
\end{cases}
\end{equation}
for some  sufficiently large constant $M>0$.
\end{lemma}

\begin{proof}
By \eqref{e4.9}, \eqref{gamma} and \eqref{e4.10}, it suffices to show
\begin{align*}
\big\|\sup_{t>0}|f\ast \KKone_{t}|\big\|_{L^{2}(\mathbb{R}^{d+1})}
\leq C(j+k)2^{j/2}C_{j,k}(\theta)||f||_{L^{2}(\mathbb{R}^{d+1})}.
\end{align*}
We may write
$$\sup_{t>0}|f\ast \KKone_{t}|=\sup_{n\in \mathbb{Z}}\sup_{t\in[1,2]}|f\ast \KKone_{2^{n}t}|.$$
By Lemma \ref{le2.2} with $F_{n}(x,u,t)=f\ast \KKone_{2^{n}t}(x,u)$ to see  that Lemma~\ref{le4.1} follows
 from the following estimates which are uniform in $t\in [1,2]$:
\begin{align}\label{e4.13}
\left(\sum_{n}\left\|f\ast \KKone_{2^{n}t}\right\|_{L^{2}(\mathbb{R}^{d+1})}^{2}\right)^{1/2}
\leq C(j+k)C_{j,k}(\theta)\|f\|_{L^{2}(\mathbb{R}^{d+1})},
\end{align}
\begin{align}\label{e4.14}
\left(	\sum_{n\in \mathbb{Z}}\left\|f\ast \left[s\frac{d}{ds}\KKone_s\right]_{s=2^{n}t}\right\|_{L^{2}(\mathbb{R}^{d+1})}^{2} \right)^{1/2}
\leq C(j+k)2^{j}C_{j,k}(\theta)\|f|_{L^{2}(\mathbb{R}^{d+1})}.
\end{align}

Next we apply an almost orthogonality    Lemma~\ref{le2.3} to prove
 \eqref{e4.13} and \eqref{e4.14}. One sees that  the inequality \eqref{e4.13}
follows from  the following estimates \eqref{e4.15} and \eqref{e4.16} if we apply a scaling argument and Lemma~\ref{le2.3}
with $A=2^{-(1-\theta)\frac{d-1}{2}j}2^{-\theta k}$ and $B= 2^{2j+2k}$:
 	\begin{align}\label{e4.15}
\|f\ast \KKone\|_{L^{2}(\mathbb{R}^{d+1})}\leq C C_{j,k}(\theta)||f||_{L^{2}(\mathbb{R}^{d+1})}, \ \ \ \ \forall
 		\theta\in[0,1]
 	\end{align}
  and
	\begin{align}\label{e4.16}
		\|f\ast  (\KKone_{2^{n}})^{\ast} \ast \KKone_{2^{n'} }\|_{L^{2}(\mathbb{R}^{d+1})}\leq C  2^{4j+4k}2^{-|n'-n|}\|f\|_{L^{2}(\mathbb{R}^{d+1})}
	\end{align}
 first for $n\leq n'$ and then by taking adjoints also for $n>n'$;
The inequality \eqref{e4.14}
follows from  the following estimates \eqref{e4.17} and \eqref{e4.18} if we apply Lemma~\ref{le2.3}
with $A=2^{j}C_{j,k}(\theta)$ and $B=2^{3(j+k)}$:

\begin{align}\label{e4.17}
	\left\|f\ast \left[s\frac{d}{ds}\KKone_s\right]_{s=2^{n}t}\right\|_{L^{2}(\mathbb{R}^{d+1})}
	\leq C2^{j}C_{j,k}(\theta)\|f\|_{L^{2}(\mathbb{R}^{d+1})},
\end{align}
 and
\begin{align}\label{e4.18}
	\left\|f\ast \left(\left[s\frac{d}{ds}\KKone_s\right]_{s=2^{n}t}\right)^{\ast}
	\ast \left[s\frac{d}{ds}\KKone_s\right]_{s=2^{n'}t} \right\|_{L^{2}(\mathbb{R}^{d+1})}
	\leq C2^{6j+6k}2^{-|n-n'|}||f||_{L^{2}(\mathbb{R}^{d+1})},
\end{align}
which are uniform in $t\in[1,2]$.
The proof of \eqref{e4.15}, \eqref{e4.16}, \eqref{e4.17} and  \eqref{e4.18} will be given in Sections~4.1.2
and  ~4.1.3 below, which are obtained by  using the cancellation of the kernels  of $\KKone$ and $s\frac{d}{ds}\KKone_s$
to show almost orthogonality properties
for these operators and certain estimates for oscillatory integrals to establish the decay estimates.
\end{proof}

\subsubsection{Reduction to oscillatory integral operators}

The following lemma is due to H\"ormander \cite{H}. Here we prove it again to show the upper bound of
the operator norm is irrelevant to the derivative of order $2$ of the phase function.
\begin{lemma}\label{le4.2}
	Define
	\begin{align*}
		T_{\lambda}f(x)=\int e^{i\lambda\Phi(x,y)}a(x,y)f(y)dy
	\end{align*}
with $\Phi(x,y)\in C^{\infty}(\mathbb{R}^{d}\times \mathbb{R}^{d})$ and $a(x,y)\in C_{c}^{\infty}(\mathbb{R}^{d}\times \mathbb{R}^{d})$.
Suppose
\begin{align}\label{e4.19}
\inf_{|X|= 1,(x,y)\in \supp a}|\partial^2_{x,y}\Phi(x,y)X|\geq c>0
\end{align}
for some constant $c>0$ and for each $|\alpha|+|\beta|\geq 3$,
\begin{align}\label{e4.20}
\sup_{(x,y)\in \supp a}|\partial_{x}^{\alpha}\partial_{y}^{\beta}\Phi(x,y)|<\infty.
\end{align}
Then there exists a constant   $C>0$  independent of $\lambda$ such that
\begin{align}\label{e4.21}
||T_{\lambda}f||_{L^{2}(\mathbb{R}^{d})}\leq C (1+|\lambda|)^{-\frac{d}{2}}||f||_{L^{2}(\mathbb{R}^{d})}.
\end{align}
\end{lemma}

\begin{proof}
If $|\lambda|\leq1$,    \eqref{e4.21}   can be obtained
by H\"older's inequality since $a(x,y)$ is a compactly supported function. For all $|\lambda|\geq1$,   it suffices   (by $T^{\ast}T$ method)
  to show that
\begin{align}\label{e4.22}
	||T_{\lambda}^{\ast}T_{\lambda}f||_{L^{2}(\mathbb{R}^{d})}\leq C|\lambda|^{-d}||f||_{L^{2}(\mathbb{R}^{d})}.
\end{align}
We may write
\begin{align*}
T_{\lambda}^{\ast}T_{\lambda}f(x)=\int_{\mathbb{R}^{d}} K_{\lambda}(x,z)f(z)dz,
\end{align*}
where
\begin{align*}
K_{\lambda}(x,z)=\int_{\mathbb{R}^{d}}  e^{-i\lambda(\Phi(y,x)-\Phi(y,z))}\overline{a}(y,x)a(y,z)dy.
\end{align*}
By a unity of participation and Schur's lemma, the proof of \eqref{e4.22} reduces  to show if $|x-z|$ is small, then
\begin{align}\label{e4.23}
|K_{\lambda}(x,z)|\leq C_{N}(1+|\lambda||x-z|)^{-N}
\end{align}
for a sufficiently large integer $N\geq1$.
Notice that
\begin{align*}
|\nabla_{y}\Phi(y,x)-\nabla_{y}\Phi(y,z)-\partial_{y,x}^{2}\Phi(y,x)(x-z)|\leq C|x-z|^{2}
\end{align*}
for some constant $C>0$. Then by the triangle inequality and \eqref{e4.19},
\begin{align*}
|\nabla_{y}\Phi(y,x)-\nabla_{y}\Phi(y,z)|\geq c|x-z|-C|x-z|^{2}\geq \frac{c}{2}|x-z|
\end{align*}
if $|x-z|\leq  {c}/{2C}$.
On another hand, it follows by \eqref{e4.20} that for all $|\alpha|\geq 1, $
$$
|\partial_{y}^{\alpha}(\nabla_{y}\Phi(y,x)-\nabla_{y}\Phi(y,z))|\leq C_{\alpha}|x-z|.
$$
By integration by parts, if $|x-z|\leq  {c}/{2C}$, then
\begin{align*}
|K_{\lambda}(x,z)|\leq C_{N} |\lambda|^{-N}|x-z|^{-N}
\end{align*}
for a sufficiently large integer $N\geq1$.
Together with the trivial estimate $|K_{\lambda}(x,z)|\leq C$,  we finish the proof of \eqref{e4.23}
and thus \eqref{e4.22} follows readily. This completes the proof of Lemma~\ref{le4.2}.
\end{proof}

Now we begin to
 verify   estimates  \eqref{e4.15}, \eqref{e4.16}, \eqref{e4.17} and  \eqref{e4.18}. To do it,
we will reduce them to  estimates for oscillatory integral operators. Recall that
  $E_{d-1}$  and  $P_{d-1}$ are two  $d\times (d-1)$ matrices given as in \eqref{e3.1}.
One can write
\begin{eqnarray*}
	f\ast K^{j,k}(x, u)=\int_{\mathbb R^{d+1}} K^{j,k}(x-y,u-v+x^{T}Jy)f(y,v)dydv,
\end{eqnarray*}
where
\begin{eqnarray*}
	K^{j,k}(x-y,u-v+x^{T}Jy)
	=\chi_{1}(x-y)\iint_{\R\times\R} e^{i(\sigma(x_{1}-y_{1}-\Gamma_1(x'-y')) + \tau (u-v+x^{T}Jy) )}
	\phi_{j}(\sigma)\phi_{k}(\tau)d\sigma d\tau.
\end{eqnarray*}
We then apply a Fourier transform on $\mathbb R$, in the $u$ variable of  $f\ast K^{j,k}(x, u)$,
to obtain
\begin{align}\label{e4.24}
 \mathcal{F}_{u}(f\ast K^{j,k})(x, \lambda_{2})
 	=2^{j}\phi_{k}(\lambda_{2})\int_{\R^{d+1}} e^{i2^{j}\sigma(x_{1}-y_{1}-\Gamma_1(x'-y'))} e^{-i\lambda_{2} x^{T}Jy}\chi_{1}(x-y)
\phi(\sigma)\mathcal{F}_{v}^{-1}f(y,\lambda_{2})dyd\sigma.
\end{align}
Define for $|\lambda_{1}|\neq1$ and $\lambda_{2}\in\R$,
\begin{align*}
T_{\lambda_{1},\lambda_{2}}g(x)&=\lambda_{1}\int_{\R^{d+1}} \chi_{1}(x-y)e^{i\lambda_{1}\sigma(x_{1}-y_{1}-\Gamma_1(x'-y'))}
e^{-i\lambda_{2}x^{T}Jy}\phi(\sigma)g(y)dyd\sigma.
\end{align*}
If $|\lambda_{1}|=1$, we should replace $\phi$ by $\varphi$.
Then we have the following result.
\begin{lemma}\label{le4.3}
For all $\theta\in[0,1]$,  there exists a constant $C>0$ such that
\begin{align}\label{e4.25}
||T_{\lambda_{1},\lambda_{2}}g||_{L^{2}(\mathbb{R}^{d})}\leq
\left\{
\begin{array}{lll} C
(1+|\lambda_{1}|)^{-(1-\frac{d}{2}\theta)\frac{d-2}{2}}(1+|\lambda_{2}|)^{-\theta}||g||_{L^{2}(\mathbb{R}^{d})} &
\textrm{\ if\ } \ M|\lambda_{2}|\geq |\lambda_{1}|;\\[6pt]
  C
 (1+|\lambda_{1}|)^{-\frac{d-1}{2}}||g||_{L^{2}(\mathbb{R}^{d})} & \textrm{\ if\ } \ M|\lambda_{2}|\leq |\lambda_{1}|
 \end{array}
 \right.
\end{align}
for some sufficiently large $M>0$.
\end{lemma}

\begin{proof}
If $|\lambda_{1}|=1$, \eqref{e4.25} can be proven by applying Lemma \ref{le4.2} in the first $k$ ($k\geq2$) variables.
It remains to show \eqref{e4.25} if $|\lambda_{1}|\neq1$.
Let us first prove \eqref{e4.25} for $M|\lambda_{2}|\geq |\lambda_{1}|$ where $M>0$ is a constant  chosen later. In this case, we  write
\begin{align*}
T_{\lambda_{1},\lambda_{2}}g(x)=\lambda_{1}\int_{\R^{d+1}} \chi_{1}(x-y)e^{i\lambda_{1}\Psi_{\lambda_{1},\lambda_{2}}(x,y)}\phi(\sigma)g(y)dyd\sigma,
\end{align*}
where the phase function $\Psi_{\lambda_{1},\lambda_{2}}$ is given by
$$
\Psi_{\lambda_{1},\lambda_{2}}(x,y)=\sigma(x_{1}-y_{1}-\Gamma_1(x'-y'))-\lambda_{1}^{-1}\lambda_{2} x^{T}Jy.
$$
We have
\begin{align*}
\partial_{x,y}^{2}\Psi_{\lambda_{1},\lambda_{2}}(x,y)= \sigma P_{d-1}\Gamma''_1(x'-y')P_{d-1}^{T}-\lambda_{1}^{-1}\lambda_{2}J,
\end{align*}
where $P_{d-1}$ is given in \eqref{e3.1}.
We apply  \eqref{e3.4} of Lemma \ref{le3.1}  by taking $Q=\partial_{x'}^{2}\Gamma_1(x'-y')$ and  $a=-\sigma^{-1}\lambda_{1}^{-1}\lambda_{2}$
to see that there exists a $c_{1}>0$ such that
\begin{align*}
|\partial_{x,y}^{2}\Psi_{\lambda_{1},\lambda_{2}}(x,y)X|
=\sigma|(P_{d-1}\partial_{x'}^{2}\Gamma_1(x'-y')P_{d-1}^{T}-\sigma^{-1}\lambda_{1}^{-1}\lambda_{2}J)X|\geq c_{1}>0,
\end{align*}
for all $|X|=1$, $x-y\in\supp \chi_{1}$ and $\sigma\in\supp \phi$.
Notice that
\begin{align*}
|\partial_{x,y}^{\alpha}\Psi_{\lambda_{1},\lambda_{2}}(x,y)|
= |-\sigma\partial_{x,y}^{\alpha}\Gamma_1(x'-y')|\leq C_{\alpha}\  \textrm{\ for all } \ |\alpha|\geq3.
\end{align*}
Define
\begin{align*}
T_{\lambda_{1},\lambda_{2}}^{\sigma}g(x)=\lambda_{1}\int \chi_{1}(x-y)e^{i\lambda_{1}\Psi_{\lambda_{1},\lambda_{2}}(x,y)}
g(y)dy.
\end{align*}
By Fubini's theorem and H\"older's inequality, we have
$$|T_{\lambda_{1},\lambda_{2}}g(x)|=\big|\int T_{\lambda_{1},\lambda_{2}}^{\sigma}g(x)\phi(\sigma)d\sigma\big|
\leq C \left(\int \big|T_{\lambda_{1},\lambda_{2}}^{\sigma}g(x)\big|^{2}\phi(\sigma)d\sigma\right)^{1/2}.$$
These, together with Lemma \ref{le4.2}, tell us that
\begin{align}\label{e4.27}
\big\|\chi_{1}(x)T_{\lambda_{1},\lambda_{2}}g(x)\big\|_{L^{2}_{x}(\mathbb{R}^{d})}&\leq C
\left(\int\big\|\chi_{1}(x)T_{\lambda_{1},\lambda_{2}}^{\sigma}g(x)\big\|_{L^{2}_{x}(\mathbb{R}^{d})}^{2}\phi(\sigma)d\sigma\right)^{1/2}\nonumber\\
&\leq C
|\lambda_{1}
|\cdot(1+|\lambda_{1}|)^{-\frac{d}{2}}\left(\int||g(y)||_{L^{2}_{y}(\mathbb{R}^{d})}^{2}\phi(\sigma)d\sigma\right)^{1/2}\nonumber\\
&\leq C (1+|\lambda_{1}|)^{-\frac{d-2}{2}}||g||_{L^{2}(\mathbb{R}^{d})}.
\end{align}

We rewrite
\begin{align*}
T_{\lambda_{1},\lambda_{2}}g(x)=\lambda_{1}\int_{\R^{d+1}}
\chi_{1}(x-y)e^{i\lambda_{2}\widetilde{\Psi}_{\lambda_{1},\lambda_{2}}(x,y)}\phi(\sigma)g(y)dyd\sigma,
\end{align*}
where $\widetilde{\Psi}_{\lambda_{1},\lambda_{2}}(x,y)=\lambda_{1}\lambda_{2}^{-1}\sigma(x_{1}-y_{1}-\Gamma_1(x'-y'))- x^{T}Jy$.
Then we have $|\partial_{x,y}^{\alpha}\widetilde{\Psi}_{\lambda_{1},\lambda_{2}}(x,y)|\leq C_{\alpha}$ for all $|\alpha|\geq2$ and
\begin{align*}
\partial^2_{(x_{1},...,x_{k}),(y_{1},...,y_{k})}\widetilde{\Psi}_{\lambda_{1},\lambda_{2}}(x,y)
= \sigma \lambda_{1}\lambda_{2}^{-1} E_{k}^{T}P_{d-1}\partial_{x'}^{2}\Gamma_1(x'-y')P_{d-1}^{T}E_{k}-J_{k}.
\end{align*}
We apply \eqref{e3.3} of Lemma \ref{le3.1} with $G=E_{k}^{T}P_{d-1}\partial_{x'}^{2}\Gamma_1(x'-y')P_{d-1}^{T}E_{k}$
and $a=-\sigma \lambda_{1}\lambda_{2}^{-1}$,
to obtain
\begin{align*}
|\partial^2_{(x_{1},...,x_{k}),(y_{1},...,y_{k})}\widetilde{\Psi}_{\lambda_{1},\lambda_{2}}(x,y)X|=|(aG+J_{k})X|\geq c_{0}>0
 \textrm{\ for all $|X|=1$, $X\in \mathbb{R}^{k}$.}
\end{align*}
By applying Lemma \ref{le4.2} in the first $k (k\geq 2)$ variables and Fubini's theorem, we have
\begin{align}\label{e4.28}
\|\chi_{1}(x)T_{\lambda_{1},\lambda_{2}}g(x)\|_{L^{2}_{x}(\mathbb{R}^{d})}
\leq C
|\lambda_{1}|\cdot (1+|\lambda_{2}|)^{-1}\|g\|_{L^{2}(\mathbb{R}^{d})}.
\end{align}
By \eqref{e4.27} and \eqref{e4.28}, for all $\theta\in[0,1]$, we have
\begin{align*}
||\chi_{1} T_{\lambda_{1},\lambda_{2}}g||_{L^{2}(\mathbb{R}^{d})}
\leq C (1+|\lambda_{1}|)^{-(\frac{d-2}{2}-\frac{d}{2}\theta)}(1+|\lambda_{2}|)^{-\theta}||g||_{L^{2}(\mathbb{R}^{d})}.
\end{align*}
By translation invariance argument similarly as in \cite[p. 236]{St2},
\begin{align*}
||T_{\lambda_{1},\lambda_{2}}g||_{L^{2}(\mathbb{R}^{d})}\leq C (1+|\lambda_{1}|)^{-(\frac{d-2}{2}-\frac{d}{2}\theta)}
(1+|\lambda_{2}|)^{-\theta}||g||_{L^{2}(\mathbb{R}^{d})}.
\end{align*}

Next we consider the case $M|\lambda_{2}|\leq |\lambda_{1}|$. If $|\lambda_{1}|\leq1$, we have
\begin{align*}
|T_{\lambda_{1},\lambda_{2}}g(x)|\leq C \chi_{1}\ast|g|(x).
\end{align*}
This implies that
$
\|T_{\lambda_{1},\lambda_{2}}g\|_{L^{2}(\mathbb{R}^{d})}\leq C \|\chi_{1}\|_{L^{1}(\mathbb{R}^{d})}\|g\|_{L^{2}(\mathbb{R}^{d})}.
$
Next we consider $|\lambda_{1}|\geq1$ and use the method in \cite[Chapter XI, p. 500]{St2}.
Choose a function $\psi\in C^{\infty}_{c}(\mathbb{R})$ with $\supp \psi\subseteq [1/2,2]$ and equal to $1$ on $[\frac{3}{4},\frac{5}{4}]$. Define
\begin{align*}
{\mathscr A}_{\lambda_{1},\lambda_{2}}g(x,\eta)=\lambda_{1}\psi(\eta)\int \chi_{1}(x-y)e^{i\lambda_{1}\Psi_{\lambda_{1},\lambda_{2}}(x,\eta,y,\sigma)}
\phi(\sigma)g(y)dyd\sigma,
\end{align*}
where
\begin{align*}
\Psi_{\lambda_{1},\lambda_{2}}(x,\eta,y,\sigma)=\sigma \eta(x_{1}-y_{1}-\Gamma_1(x'-y'))-\lambda_{1}^{-1}\lambda_{2} x^{T}Jy.
\end{align*}
Observe that
\begin{align*}
\partial_{(x,\eta),(y,\sigma)}^{2}\Psi_{\lambda_{1},\lambda_{2}}(x,\eta,y,\sigma)
=\left(
\begin{array}{cccccc}
0 & 0& \eta \\
0 & \sigma \eta\partial_{x'}^{2}\Gamma_1 & \eta(\nabla_{x'}\Gamma_1)^{T}\\
\sigma & \sigma\nabla_{x'}\Gamma_1 & \partial_{\eta,\sigma}^{2}\Psi_{\lambda_{1},\lambda_{2}}\\
\end{array}
\right)-
\lambda_{1}^{-1}\lambda_{2}\left(\begin{array}{cccccc}
J & 0 \\
0 & 0\\
\end{array}
\right),
\end{align*}
which implies
$$
\det\partial_{(x,\eta),(y,\sigma)}^{2}\Psi_{\lambda_{1},\lambda_{2}}(x,\eta,y,\sigma)
=(\sigma \eta)^{d}\det\partial_{x'}^{2}\Gamma_1+O(
\lambda_{1}^{-1}\lambda_{2})
$$
and
\begin{align*}
|\partial_{x,\eta,y,\sigma}^{\alpha}\Psi(x,\eta,y,\sigma)|\leq C_{\alpha} \textrm{\ for all $|\alpha|\geq2$}.
\end{align*}
Then by \eqref{e4.3}, for all $x-y\in\supp \chi_{1}$, $\sigma\in\supp \phi$ and $\eta\in\supp \psi$, we have
\begin{align*}
|\det\partial_{(x,\eta),(y,\sigma)}^{2}\Psi_{\lambda_{1},\lambda_{2}}(x,\eta,y,\sigma)|&\geq
c_{H}^{(1)}(\sigma \eta)^{d}
-C\lambda_{1}^{-1}\lambda_{2}
 \geq c_{H}^{(1)}4^{-d}
-CM^{-1}\geq \frac{c_{H}^{(1)}}{2}4^{-d},
\end{align*}
where we choose $M:= {4^{d+1}C}{[c_{H}^{(1)}]^{-1}}$ and $c_{H}^{(1)}$ is the one defined in \eqref{e4.3}.
Thus it follows from Lemma \ref{le4.2} that
\begin{align*}
||\chi_{1}(x) {\mathscr A}_{\lambda_{1},\lambda_{2}}g(x,\eta)||_{L^{2}(\mathbb{R}^{d+1})}
\leq& C |\lambda_{1}|(1+|\lambda_{1}|)^{-\frac{d+1}{2}}||\phi(\sigma) g(y)||_{L^{2}(\mathbb{R}^{d+1})}\\
\leq& C (1+|\lambda_{1}|)^{-\frac{d-1}{2}}||g||_{L^{2}(\mathbb{R}^{d})}.
\end{align*}
By translation invariance argument again,
\begin{align}\label{e4.29}
\sup\limits_{\eta\in[1/2,2]}||{\mathscr A}_{\lambda_{1},\lambda_{2}}g(x,\eta)||_{L^{2}(\mathbb{R}^{d+1})}
\leq C (1+|\lambda_{1}|)^{-\frac{d-1}{2}}||g||_{L^{2}(\mathbb{R}^{d})}.
\end{align}
Since
\begin{align*}
\frac{d}{d\eta} {\mathscr A}_{\lambda_{1},\lambda_{2}}g(x,\eta)&=-\lambda_{1}\psi(\eta)\int \chi_{1}(x-y)e^{i\lambda_{1}\Psi(x,\eta,y,\sigma)}
\frac{d}{d\sigma}(\frac{\sigma}{\eta}\phi(\sigma))g(y)dyd\sigma \nonumber\\
&+\lambda_{1}\psi'(\eta)\int \chi_{1}(x-y)e^{i\lambda_{1}\Psi(x,\eta,y,\sigma)}\phi(\sigma)g(y)dyd\sigma,
\end{align*}
it is seen that
\begin{align}\label{e4.30}
\sup\limits_{\eta\in[1/2,2]}\big\|\frac{d}{d\eta} {\mathscr A}_{\lambda_{1},\lambda_{2}}g(x,\eta)\big\|_{L^{2}(\mathbb{R}^{d+1})}
\leq C (1+|\lambda_{1}|)^{-\frac{d-1}{2}}||g||_{L^{2}(\mathbb{R}^{d})}.
\end{align}
By   H\"older's inequality, we have
$$\sup\limits_{\eta\in[1,2]}|{\mathscr A}_{\lambda_{1},\lambda_{2}}g(x,\eta)|^{2}
\leq 2\bigg(\int_{1}^{2}|{\mathscr A}_{\lambda_{1},\lambda_{2}}g(x,\eta)|^{2}d\eta\bigg)^{1/2}
\bigg(\int_{1}^{2}|\frac{d}{d\eta} {\mathscr A}_{\lambda_{1},\lambda_{2}}g(x,\eta)|^{2}d\eta\bigg)^{1/2}
+|{\mathscr A}_{\lambda_{1},\lambda_{2}}g(x,1)|^{2},$$
which, together with \eqref{e4.29}, \eqref{e4.30} and   H\"older's inequality in $x$, implies us that
\begin{align*}
\|T_{\lambda_{1},\lambda_{2}}g\|_{L^{2}(\mathbb{R}^{d})}=\|{\mathscr A}_{\lambda_{1},\lambda_{2}}g(x,1)\|_{L^{2}(\mathbb{R}^{d})}
\leq C (1+|\lambda_{1}|)^{-\frac{d-1}{2}}||g||_{L^{2}(\mathbb{R}^{d})}.
\end{align*}
This finishes the proof of \eqref{e4.25}.
\end{proof}

\subsubsection{Proof of \eqref{e4.15} and \eqref{e4.16} }

\begin{proof} [{\rm PROOF OF \eqref{e4.15}}]
By Plancherel's theorem, \eqref{e4.24}, Fubini's theorem and Lemma~\ref{le4.3}, we have
\begin{align*}
\|f\ast K^{j,k}\|_{L^{2}(\mathbb{R}^{d+1})}&=\|\mathcal{F}_{u}(f\ast K^{j,k})\|_{L^{2}(\mathbb{R}^{d+1})}\\
&=\big\|T_{2^{j}, \lambda_2}\big(\mathcal{F}_{v}f(\cdot, \lambda_2)\big)(x)\big\|_{L^{2}_{\lambda_2}(\mathbb{R}^{2})L^{2}_{x}(\mathbb{R}^{d})}\nonumber\\
&\leq C 2^{-(1-\frac{d}{2}\theta)\frac{d-2}{2}j}2^{-\theta k}
\|\mathcal{F}_{v}f(\cdot, \lambda_2)\|_{L^{2}_{ \lambda_2}(\mathbb{R}^{2})L^{2}_{x}(\mathbb{R}^{d})}\\
&\leq C2^{-(1-\frac{d}{2}\theta)\frac{d-2}{2}j}2^{-\theta k}\|f\|_{L^{2}(\mathbb{R}^{d+1})}
\end{align*}
for all $j,k\geq 0 $ and $\theta\in[0,1]$. This, together with \eqref{e4.9}, yields that \eqref{e4.15} holds.

\end{proof}

\begin{proof}[{\rm PROOF OF \eqref{e4.16}}]
By a scaling argument, the proof of \eqref{e4.16} reduces to show for all $n\leq0$,
\begin{align}\label{e4.31}
	\left\|f\ast  (\KKone  )^{\ast} \ast \KKone_{2^{n}}  \right\|_{L^{2}(\mathbb{R}^{d+1})}\leq C  2^{4j+4k}2^{n}\|f\|_{L^{2}(\mathbb{R}^{d+1})}
\end{align}
and \begin{align}\label{e4.32}
	\left\|f\ast  (\KKone_{2^{n}})^{\ast} \ast \KKone \right\|_{L^{2}(\mathbb{R}^{d+1})}\leq C  2^{4j+4k}2^{n}\|f\|_{L^{2}(\mathbb{R}^{d+1})}.
\end{align}

We start with the proof of \eqref{e4.31}. We use  integration by parts   to see that for all $|\alpha|\leq 1 $ and $N>0,$
there exists a constant $C=C(\alpha, N)>0$ such that
\begin{align} \label{e4.33}
|\partial_{x,u}^{\alpha}\KKone(x,u)| + |\partial_{x,u}^{\alpha}(\KKone)^{\ast}(x,u)|
\leq C 2^{2j+2k}\chi_{1}(x)\bigg(1+|u|\bigg)^{-N}.
\end{align}
Note that  $|x''^{T}E_{d-1}^{T}JE_{d-1}y''|\leq2c\|J\|\cdot|y''|\leq 2c^{2}\|J\|$ if $|x''|,|y''|\leq c$, where
$x''=(x_{1},x_{2},....,x_{d-1})$ and $y''=(y_{1},y_{2},....,y_{d-1})$.
These, in combination the condition  $x^{T}Jy=x''^{T}E_{d-1}^{T}JE_{d-1}y''$
and the cancellation property \eqref{e4.11} of the kernel of $\KKone$, yield
\begin{eqnarray*}
|(\KKone)^{\ast}\ast \KKone_{2^{n}}(x,u)|
 &=&\big|\int_{\mathbb R^{d+1}} (\KKone)^{\ast}(x-y,u-v-x^{T}Jy)\KKone_{2^{n}}(y,v) dydv\big|\nonumber\\
&=&\big|\int_{\mathbb R^{d+1}} \bigg((\KKone)^{\ast}(x-y,u-v-x''^{T}E_{d-1}^{T}JE_{d-1}y'')-(\KKone)^{\ast}(x,u)\bigg) \KKone_{2^{n}}(y,v) dydv\big|\\
 &\leq&
 \int_{0}^{1}\int_{\R^{d+1}}2^{2j+2k}\bigg(1+|x-\theta y|+|u-\theta v|\bigg)^{-N}
\big(|y|+|v|\big)
 \big|\KKone_{2^{n}}(y,v)\big|dydvd\theta.
\end{eqnarray*}
This gives that for $n\leq 0$,
\begin{eqnarray*}
 \int_{\mathbb R^{d+1}}|(\KKone)^{\ast}\ast \KKone_{2^{n}}(x,u)|dxdu
&\leq& C 2^{2j+2k}\int_{\mathbb R^{d+1}}\big(|y|+|v|\big)\big|\KKone_{2^{n}}(y,v)\big|dydv\nonumber\\
&\leq& C^{2}2^{4j+4k}2^{n}.
\end{eqnarray*}
By the Schur lemma, \eqref{e4.31} follows readily.

Next we prove \eqref{e4.32}.
We use the cancellation property of $(\KKone)^\ast$
and a similar argument as above to show that
$$
\int_{\mathbb R^{d+1}} |(\KKone_{2^{n}})^{\ast}\ast \KKone(x,u)|dxdu\leq C 2^{4j+4k}2^{n}.
$$
This, together with the Schur lemma,  gives the desired estimate \eqref{e4.32}.
\end{proof}

\subsubsection{Proof of \eqref{e4.17} and  \eqref{e4.18}}

\begin{proof} [{\rm PROOF OF \eqref{e4.17}}]
  It suffices to show
\begin{align*}
	\left\|f\ast \left[s\frac{d}{ds}\Kone_s\right]_{s=2^{n}t}\right\|_{L^{2}(\mathbb{R}^{d+1})}
	\leq C2^{j}C_{j,k}(\theta)\|f\|_{L^{2}(\mathbb{R}^{d+1})}.
\end{align*}
By \eqref{e4.4} and change of variables in $\sigma,\tau$, we have
\begin{align*}
\Kone_{s}(x,u)
 =s^{-d+1}\chi_{1}(s^{-1}x)\iint_{\R\times\R} e^{i\sigma(x_{1}-s\Gamma_1(s^{-1}x'))}e^{i\tau u}
\phi_{j}(s\sigma)\phi_{k}(s^{2}\tau)d\sigma d\tau.
\end{align*}
Observe that
\begin{eqnarray}\label{e4.34}
s\frac{d}{ds}\Kone_s=(K^{(1)})^{j,k}_s  +  (K^{(2)})^{j,k}_s  +2^{j}(K^{(3)})^{j,k}_s
\end{eqnarray}
where
\begin{eqnarray*}
(K^{(1)})^{j,k}_s &=&
\left[(-d+1)s^{-d+1}\chi_{1}(s^{-1}x) - s^{-1}x\chi_{1}'(s^{-1}x) \right]    s^{-d+1}\iint_{\R\times\R}e^{i\sigma(x_{1}-s\Gamma_1(s^{-1}x'))}e^{i\tau u}
\phi_{j}(s\sigma)\phi_{k}(s^{2}\tau)d\sigma d\tau,\\
 (K^{(2)})^{j,k}_s &=&
s^{-d+1}\chi_{1}(s^{-1}x)\iint_{\R\times\R} e^{i\sigma(x_{1}-s\Gamma_1(s^{-1}x'))}e^{i\tau u}\left[2^{-j}s\sigma
\phi'_{j}(s\sigma)\phi_{k}(s^{2}\tau)  +2^{-k+1} s^{2}\tau \phi_{j}(s\sigma)\phi'_{k}(s^{2}\tau)\right]  d\sigma d\tau,
\\
  (K^{(3)})^{j,k}_s&=& \left[is^{-1}x\Gamma_1'(s^{-1}x')   -i\Gamma_1(s^{-1}x')\right]     s^{-d+1}\chi_{1}(s^{-1}x)
   \iint_{\R\times\R} e^{i\sigma(x_{1}-s\Gamma_1(s^{-1}x'))}e^{i\tau u}
2^{-j}s\sigma\phi_{j}(s\sigma)\phi_{k}(s^{2}\tau)d\sigma d\tau.
\end{eqnarray*}

By Plancherel's theorem, Fubini's theorem and Lemma~\ref{le4.3}, an argument as in the proof of \eqref{e4.15} yields that  for all $i=1, 2, 3$,
\begin{align*}
\|f\ast (K^{(i)})^{j,k}\|_{L^{2}(\mathbb{R}^{d+1})}\leq C C_{j,k}(\theta)\|f\|_{L^{2}(\mathbb{R}^{d+1})}, \ \ \ \ \forall
 \theta\in[0,1],
 \end{align*}
which, together with a scaling argument, implies that
\begin{align*}
\sup_{t\in[1,2]}\big\|f\ast \left[s\frac{d}{ds}\Kone_s\right]_{s=2^{n}t}\big\|_{L^{2}(\mathbb{R}^{d+1})}
&\leq 2^{j}\sum_{i=1}^{6}\sup_{t\in[1,2]}\big\| f\ast (K^{(i)})^{j,k}_{2^{n}t}\big\|_{L^{2}(\mathbb{R}^{d+1})}\nonumber\\
&\leq  C2^{j}C_{j,k}(\theta)\|f\|_{L^{2}(\R^{d+1})}.
\end{align*}
This completes the proof of \eqref{e4.17}.
\end{proof}

\begin{proof} [{\rm PROOF OF \eqref{e4.18}}]
 The proof is similar as \eqref{e4.16} and thus we only sketch it. By \eqref{e4.11}, $s\frac{d}{ds}\KKone_s$ and $\big(s\frac{d}{ds}\KKone_s\big)^{*}$
also have   integral zero. By \eqref{e4.34} and \eqref{e4.9}, an argument as in \eqref{e4.33}, we see that for all $|\alpha|\leq 1 $ and $N>0$,
there exists a constant $C=C(\alpha, N)>0$ such that
\begin{align*}
\left|\partial_{x}^{\alpha_{1}}\partial_{u}^{\alpha_{2}}\left(s\frac{d}{ds}\KKone_{s}\right)(x,u)\right|
+ \left|\partial_{x}^{\alpha_{1}}\partial_{u}^{\alpha_{2}}\left(s\frac{d}{ds}\KKone_{s}\right)^{\ast}(x,u)\right|
\leq C 2^{3j+3k}s^{-d-2-|\alpha_{1}|-2|\alpha_{2}|}\chi_{1}(s^{-1}x)\big(1+s^{-2}|u|\big)^{-N},
\end{align*}
where $\alpha=(\alpha_{1},\alpha_{2})$.
It suffices to show for all $n,n'\in \mathbb{Z}$,
\begin{align}\label{527-e4.34}
\int_{\R^{d+1}}\bigg| \left(\left[s\frac{d}{ds}\KKone_s\right]_{s=2^{n}t}\right)^{\ast}
	\ast \left[s\frac{d}{ds}\KKone_s\right]_{s=2^{n'}t}(x,u)\bigg|dxdu
	\leq C2^{6j+6k}2^{-|n-n'|}||f||_{L^{2}(\mathbb{R}^{d+1})},
\end{align}
\eqref{e4.18} follows. By the cancelation property of $\left[s\frac{d}{ds}\KKone_s\right]_{s=2^{n}t}$
if $n\geq n'$ , $\left[s\frac{d}{ds}\KKone_s\right]_{s=2^{n'}t}$ if $n'\geq n$ and mean value theorem, \eqref{527-e4.34} follows.
\end{proof}

\subsection{ Weak type $(1,1)$-estimate } In this section, we will prove  the following result.
\begin{lemma}\label{le4.4}
Let  $j,k\geq 0$. For all $\alpha>0$, we have
\begin{align*}
\big| \big\{(x, u): \sup_{t>0}|f\ast \Kone_{t}(x, u)| > \alpha \big\} \big| \leq C (j+k)2^{j} \alpha^{-1} ||f||_{L^{1}(\mathbb{R}^{d+1})}.
\end{align*}
\end{lemma}

\begin{proof}
Let  $\rho(x,u):=|x|+|u|^{1/2}$. For all $(x,u)$, $(y,v)\in\R^{d+1}$, we have
$$
\rho ((x,u)(y, v))\leq c_{J}\left( \rho(x,u) + \rho(y,v)\right), \ \ \ \ {\rm with}\ c_{J}=\frac{3}{2}\|J\|^{1/2}+1.
$$
From this, we see that if $\rho(x,u)\geq 2c_{J}\rho(y,v)$, then
\begin{align}\label{e4.35}
  \rho\big((y,v)^{-1}(x,u)\big)\geq (2c_{J})^{-1}\rho(x,u).
\end{align}

Since we have already shown the $L^2$ bounds for the maximal function, from the Calder\'on-Zygmund theory  it suffices to verify
the following H\"ormander-type condition:
\begin{align*}
\int_{\rho(x,u)\geq 2c_{J}\rho(y,v)}\sup_{s>0}
\big|K^{j,k}_{s}\big((y,v)^{-1}(x,u)\big)-K^{j,k}_{s}(x,u)\big|dxdu\leq C (j+k)2^{j}.
\end{align*}
By a scaling argument, it suffices to show for all $\rho(y,v)\leq 2^{-n}\delta$ and $\delta>0$,
\begin{align}\label{e4.36}
\sum_{n\in \mathbb{Z}}\int_{\rho(x,u)\geq 2c_{J}2^{-n}\delta}\sup_{s\in[1,2]}
\big|K^{j,k}_{s}\big((y,v)^{-1}(x,u)\big)-K^{j,k}_{s}(x,u)\big|dxdu\leq C (j+k)2^{j}.
\end{align}

To show \eqref{e4.36}, we first note that from \eqref{e4.4}, we have that for any $j, k\geq 0,$
\begin{align}\label{e4.37}
\big|K^{j,k}(x,u)\big|\leq C2^{j}2^{k}\big(1+|x|+|u|^{1/2}\big)^{-N/2}\big(1+2^{j}|x_{1}-\Gamma_1(x')|\big)^{-N/2}\big(1+2^{k/2}|u|^{1/2}\big)^{-N/2},
\end{align}
\begin{align}\label{7-26-e4.38}
\sup_{s\in[1,2]}\big|\nabla_{x,u}K_{s}^{j,k}(x,u)\big|\leq C2^{2j}2^{2k}\big(1+|x|+|u|^{1/2}\big)^{-N/2},
\end{align}
for  large  $N\geq 1$; and by \eqref{e4.34},
\begin{align}\label{e4.38}
\big|\frac{d}{ds}\Kone_s(x,u)\big|
\leq C2^{2j}2^{k}(1+|x|+|u|^{1/2})^{-N/2} \bigg(1+2^{j}|s^{-1}x_{1}-\Gamma_1(s^{-1}x')|\bigg)^{-N/2}\big(1+2^{k/2}|u|^{1/2}\big)^{-N/2}.
\end{align}
 uniformly in $s\in[1,2]$. Now we consider two cases: $ 2^{-n}\delta\geq 1$ and $2^{-n}\delta\leq 1.$

 \smallskip

 \noindent
 {\bf Case 1: $2^{-n}\delta\geq 1$.}

   In this case,  by a triangle inequality, a change of variable, Newton-Leibniz formula,
   \eqref{e4.35}, \eqref{e4.37} and \eqref{e4.38}, we see that for all $\rho(y,v)\leq 2^{-n}\delta$,
\begin{align}\label{e4.39-520}
&\int_{\rho(x,u)\geq 2c_{J} 2^{-n}\delta}\sup_{s\in[1,2]}\big|K^{j,k}_{s}
\big((y,v)^{-1}(x,u)\big)-K^{j,k}_{s}(x,u)\big|dxdu\nonumber\\
&\leq\int_{\rho(x,u)\geq 2c_{J} 2^{-n}\delta}\sup_{s\in[1,2]}\bigg(\big|K^{j,k}_{s}
\big((y,v)^{-1}(x,u)\big)\big|+\big|K^{j,k}_{s}(x,u)\big|\bigg)dxdu\nonumber\\
&\leq
2\int_{\rho(x,u)\geq 2^{-n}\delta}\sup_{s\in[1,2]}
\big|K^{j,k}_{s}(x,u)\big|dxdu\nonumber\\
&\leq 2\int_{\rho(x,u)\geq 2^{-n}\delta}\big|\frac{d}{ds}\Kone_s(x,u)
\big|dxdu+2\int_{\rho(x,u)\geq 2^{-n}\delta}\big|K_{j,k}(x,u)\big|dxdu\nonumber\\
&\leq C(2^{-n}\delta)^{-N/2}2^{j},
\end{align}
where we use $2c_{J}\geq1$.
Then we have
\begin{align*}
\sum_{2^{-n}\delta\geq1}\int_{\rho(x,u)\geq 2c_{J}2^{-n}\delta}\sup_{s\in[1,2]}
\big|K^{j,k}_{s}\big((y,v)^{-1}(x,u)\big)-K^{j,k}_{s}(x,u)
\big|dxdu\leq C2^{j}.
\end{align*}

 \smallskip

 \noindent
 {\bf Case 2: $2^{-n}\delta\leq 1$.}

In this case, we will  show for all $\rho(y,v)\leq 2^{-n}\delta$,
\begin{align}\label{e4.39}
\sum_{2^{-n}\delta\leq1}\int_{\rho(x,u)\geq 2c_{J}2^{-n}\delta}\sup_{s\in[1,2]}
\big|K^{j,k}_{s}\big((y,v)^{-1}(x,u)\big)-K^{j,k}_{s}(x,u)\big|dxdu\leq C (j+k)2^{j}.
\end{align}
By the support condition of $K^{j,k}$, we can suppose $|x-y|$ or $|x|\leq C$. By $\rho(y,v)\leq 2^{-n}\delta\leq 1$,
we have $|y|\leq 1$ and $|v|\leq1$. As a result, one must have $|x|\leq C+1$ and
\begin{align}\label{e4.40}
|(y,v+y^{T}Jx)|\leq C\rho(y,v).
\end{align}
Let $\theta\in[0,1]$.
By \eqref{e4.35} and $\rho(x,u)\geq 2c_{J}\rho(y,v)\geq 2c_{J}\rho(\theta y,\theta v)$, we have
\begin{align*}
\rho((\theta y,\theta v)^{-1}(x,u))\geq (2c_{J})^{-1}\rho(x,u).
\end{align*}
This, together with Newton-Leibniz formula, \eqref{7-26-e4.38} and \eqref{e4.40}, tells us that
\begin{align*}
&\sup_{s\in[1,2]}\big|K^{j,k}_{s}\big((y,v)^{-1}(x,u)\big)-K^{j,k}_{s}(x,u)
\big|\nonumber\\
&\leq\sup_{s\in[1,2]}\int_{0}^{1}\big|\big(\nabla_{x,u}\Kone_s\big)\big((\theta y,\theta v)^{-1}(x,u)\big)\big|d\theta|(y,v+y^{T}Jx)|\nonumber\\
&\leq C 2^{2j+2k}\rho(y,v)\int_{0}^{1}\bigg(1+\rho((\theta y,\theta v)^{-1}(x,u))\bigg)^{-N/2}d\theta\nonumber\\
&\leq C 2^{2j+2k}2^{-n}\delta\bigg(1+\rho(x,u)\bigg)^{-N/2}.
\end{align*}
After integrating in $x,u$, we have
\begin{align}\label{e4.42}
 &\sum_{2^{-n}\delta\leq2^{-2j-2k}}\int_{\rho(x,u)\geq 2 c_{J}2^{-n}\delta}\sup_{s\in[1,2]}
\big|K^{j,k}_{s}\big((y,v)^{-1}(x,u)\big)-K^{j,k}_{s}(x,u)\big|dxdu \nonumber\\
&\leq C \sum_{2^{-n}\delta\leq2^{-2j-2k}}2^{2j+2k}2^{-n}\delta\leq C.
\end{align}
On another hand, we apply  \eqref{e4.37},
  \eqref{e4.38}   and a similar argument as in  the proof of \eqref{e4.39-520} to obtain
\begin{align}\label{e4.43}
&\sum_{2^{-2j-2k}\leq2^{-n}\delta\leq1}\int_{\rho(x,u)\geq 2c_{J}2^{-n}\delta}
\sup_{s\in[1,2]}\big|K^{j,k}_{s}\big((y,v)^{-1}(x,u)\big)-K^{j,k}_{s}(x,u)\big|dxdu\nonumber\\
&\leq\sum_{2^{-2j-2k}\leq2^{-n}\delta\leq1}\int_{\R^{d+1}}
\sup_{s\in[1,2]}\big|K^{j,k}_{s}\big((y,v)^{-1}(x,u)\big)|+\sup_{s\in[1,2]}|K^{j,k}_{s}(x,u)\big|dxdu\nonumber\\
&\leq 2 \sum_{2^{-2j-2k}\leq2^{-n}\delta\leq1}\int_{\R^{d+1}}\sup_{s\in[1,2]}
\big|K^{j,k}_{s}(x,u)\big| dxdu\nonumber\\
&\leq C \sum_{2^{-2j-2k}\leq2^{-n}\delta\leq1}2^{j}\leq C (j+k)2^{j},
\end{align}
where in the last inequality we use the fact that the summation has only
$O(j+k)$ terms. \eqref{e4.39} follows from \eqref{e4.42} and \eqref{e4.43}. Hence  \eqref{e4.36} is proved,
and then the proof of Lemma~\ref{le4.4} is completed.
\end{proof}

\subsection{ Proof of (i)  of Theorem~\ref{th2.1} }
\begin{proof}[Proof of (i)  of Theorem~\ref{th2.1} ]
First we consider $1<p\leq 2$. By \eqref{e4.6}, we have
$$\sup_{t>0}|f\ast (\mu_1)_{t}|\leq C\sum_{j=0}^{\infty} \sum_{k=0}^{\infty}\sup_{t>0}\big||f|\ast K^{j,k}_{t}\big|.$$

From \eqref{e4.7}, it suffices to consider that  $j+k\geq1$.
By interpolation, it follows from Lemmas  \ref{le4.1} and \ref{le4.4} that  for all $\theta\in[0,1]$,
$$
\bigg\|\sup_{t>0}\big||f|\ast K^{j,k}_{t}\big|\bigg\|_{L^{p}(\R^{d+1})}
\leq C (j+k) 2^{	[{1\over p}-(d-2)(1-\frac{1}{p})]j} 2^{ {1\over 4} d(d-2)\theta j}2^{-2 (1-\frac{1}{p})\theta k}\|f\|_{L^{p}(\R^{d+1})}.
$$
As a result, the summation in $j,k$ converges whenever  $ {(d-1)}/{(d-2)}<p\leq 2 $ with  $d\geq 4$  and  $\theta$ is
sufficiently small. By interpolation with a trivial $L^{\infty}$ estimate, we obtain the desired $L^{p}$
estimate in (i) of Theorem~\ref{th2.1}.
\end{proof}


\bigskip

 \section{Proof of (ii) of Theorem~\ref{th2.1}}
\setcounter{equation}{0}

In this  section, we will show that (ii) of Theorem~\ref{th2.1}. That is, for  $d\geq 3$,
the maximal operator $\sup_{t>0}|f\ast ( \mu_2)_{t} (x,u)|$ is bounded on $L^p(\mathbb R^{d+1})$
where $\mu_2$ is given as in \eqref{e2.3}, i.e.,
 \begin{eqnarray*}
\mu_2(x, u)=\chi_{\mu_2}(x,u)\iint_{\mathbb R \times \mathbb R} e^{i(\sigma(x_d- \Gamma_2(x'')) +\tau u)} d\sigma d\tau
\end{eqnarray*}
with  $x_{d}=\Gamma_2(x''),  x''= (x_{1},...x_{d-1})\ $  with  $\Gamma_2\in C^{\infty}(\mathbb R^{d-1})$
and  $\mu_2$ is supported in a small neighborhood of $(x_{0}'', \Gamma_2(x_{0}''))$ for some $x_{0}''\in \R^{d-1}$.

Since $\mu_2$ has a sufficiently small support,	
 we can choose a smooth nonnegative function $\chi_{2}$
  defined on $\mathbb{R}^{d-1}$,
which is equal to $1$ on the projection of $\supp \mu_{2}$ to $x''$-space and also has a small support.  By the curvature hypothesis and the convexity of $\Sigma$,
we can further assume $\partial_{x''}^{2}\Gamma_2(x'')$ is positive
definite for all $x''\in\supp \chi_{2}$  such that
\begin{eqnarray} \label{e5.1}
\inf_{x''\in\supp \chi_{2}}\left| {\rm det} \left(\partial_{x''}^{2}\Gamma_2(x'') \right) \right|\geq c_{H}^{(2)}
\end{eqnarray}
for some $c_{H}^{(2)}>0$. Let $\{\phi_k\}_{k=0}^{\infty} $ be given as in \eqref{e4.2}. Define for $j, k\geq 0$,
 \begin{eqnarray}\label{e5.2}
K^{j,k}(x,u)&=&\chi_{2}(x'')\iint_{\R\times\R} e^{i\sigma(x_{d}-\Gamma_2(x''))}e^{i\tau u}\phi_{j}(\sigma)\phi_{k}(\tau)d\sigma d\tau.
\end{eqnarray}
An argument as in \eqref{e4.6} yields that
\begin{align}\label{e5.3}
|f\ast (\mu_2)_{t}|\leq   C\sum_{j=0}^{\infty} \sum_{k=0}^{\infty}  |f|\ast K^{j,k}_{t}.
\end{align}

For  $j=k=0$,   we have that
 $
|K^{0,0}(x,u)|
 \leq C_N (1+|x|+|u|)^{-N}
$
for any $N>0$. From this we see that  $\sup_{t>0}|f\ast (K^{0,0})_{t}| $
  is controlled by the appropriate variant of the Hardy-Littlewood maximal function and therefore  \cite{St2}  we have the inequality
\begin{align}\label{e55}
\big|\big|\sup_{t>0}|f\ast (K^{0,0})_{t}|\big|\big|_{L^{p}(\mathbb{R}^{d+1})}
\leq C||f||_{L^{p}(\mathbb{R}^{d+1})}, \ \ \ \  1<p<\infty.
\end{align}
Consider  $j+k\geq 1$. We have that
\begin{align}\label{e5.4}
	\int_{\mathbb R^{d+1}} K^{j,k}(x,u)dxdu
	&=\int_{\mathbb R^{d+1}} \chi_{2}(x'')\mathcal{F}(\phi_{j})\big(x_{d}-\Gamma_2(x'')\big)\mathcal{F}
	(\phi_{k})(u)dxdu\nonumber\\
	&=\int_{\mathbb R^{d+1}} \chi_{2}(x'')\mathcal{F}(\phi_{j})\big(x_{d}\big)\mathcal{F}(\phi_{k})(u)dxdu=0,
\end{align}
since $\int_{\mathbb R}\mathcal{F}(\phi_{l})(s)ds=0$ if $l>0$.

\subsection{Square functions and almost orthogonality }

\begin{lemma}\label{le5.1}
	For all $j+k\geq 1 $ and $\theta\in[0,1]$, we have
	\begin{align}\label{e5.5}
		\big|\big|\sup_{t>0}|f\ast K^{j,k}_{t}|\big|\big|_{L^{2}(\mathbb{R}^{d+1})}
		\leq C (j+k)2^{-(\frac{d-2}{2}-\frac{d-1}{2}\theta)j}2^{-\theta k}||f||_{L^{2}(\mathbb{R}^{d+1})}.
	\end{align}
\end{lemma}

 \begin{proof}
To show \eqref{e5.5},  one writes
	$$
	\sup_{t>0}|f\ast K^{j,k}_{t}|=\sup_{n\in \mathbb{Z}}\sup_{t\in[1,2]}|f\ast K^{j,k}_{2^{n}t}|.
	$$
	We apply  Lemma \ref{le2.2} with $F_{n}(x,u,t)=f\ast K^{j,k}_{2^{n}t}(x,u)$
	to see  that \eqref{e5.5} follows from the following estimates which are uniform in $t\in [1,2]$:
	\begin{align}\label{e5.6}
	\left(	\sum_{n\in \mathbb{Z}}\left\|f\ast K^{j,k}_{2^{n}t}\right\|_{L^{2}(\mathbb{R}^{d+1})}^{2}\right)^{1/2}
		\leq C (j+k)2^{-(1-\theta)\frac{(d-1)}{2}j}2^{-\theta k}||f||_{L^{2}(\mathbb{R}^{d+1})}
	\end{align}
	and
\begin{align}\label{e5.7}
	\left(	\sum_{n\in \mathbb{Z}}\left\|f\ast \left[s\frac{d}{ds}K^{j,k}_s\right]_{s=2^{n}t}\right\|_{L^{2}(\mathbb{R}^{d+1})}^{2} \right)^{1/2}
		\leq C (j+k)2^{-(\frac{d-3}{2}-\frac{d-1}{2}\theta)j}2^{-\theta k}||f||_{L^{2}(\mathbb{R}^{d+1})}.
\end{align}

Next we apply an almost orthogonality    Lemma~\ref{le2.3} to prove
 \eqref{e5.6} and \eqref{e5.7}. One sees that  the inequality \eqref{e5.6}
follows from  the following estimates \eqref{e5.8} and \eqref{e5.9} if we apply a scaling argument and Lemma~\ref{le2.3}
with $A=2^{-(1-\theta)\frac{d-1}{2}j}2^{-\theta k}$ and $B= 2^{2j+2k}$:
 	\begin{align}\label{e5.8}
\|f\ast K^{j,k}\|_{L^{2}(\mathbb{R}^{d+1})}\leq C 2^{-(1-\theta)\frac{d-1}{2}j}2^{-\theta k}||f||_{L^{2}(\mathbb{R}^{d+1})}, \ \ \ \ \forall
 		\theta\in[0,1]
 	\end{align}
  and
	\begin{align}\label{e5.9}
		\|f\ast  (K^{j,k}_{2^{n}})^{\ast} \ast K^{j,k}_{2^{n'} }\|_{L^{2}(\mathbb{R}^{d+1})}\leq C  2^{4j+4k}2^{-|n'-n|}\|f\|_{L^{2}(\mathbb{R}^{d+1})}
	\end{align}
 first for $n\leq n'$ and then by taking adjoints also for $n>n'$; The inequality \eqref{e5.7}
follows from  the following estimates \eqref{e5.10} and \eqref{e5.11} if we apply Lemma~\ref{le2.3}
with $A=2^{-(\frac{d-3}{2}-\frac{d-1}{2}\theta)j}2^{-\theta k}$ and $B=2^{3(j+k)}$:

\begin{align}\label{e5.10}
	\left\|f\ast \left[s\frac{d}{ds}K^{j,k}_s\right]_{s=2^{n}t}\right\|_{L^{2}(\mathbb{R}^{d+1})}
	\leq C2^{-(\frac{d-3}{2}-\frac{d-1}{2}\theta)j}2^{-\theta k}||f||_{L^{2}(\mathbb{R}^{d+1})},
\end{align}
 and
\begin{align}\label{e5.11}
	\left\|f\ast \left(\left[s\frac{d}{ds}K^{j,k}_s\right]_{s=2^{n}t}\right)^{\ast}
	\ast \left[s\frac{d}{ds}K^{j,k}_s\right]_{s=2^{n'}t} \right\|_{L^{2}(\mathbb{R}^{d+1})}
	\leq C2^{6j+6k}2^{-|n-n'|}||f||_{L^{2}(\mathbb{R}^{d+1})}.
\end{align}

The proof of \eqref{e5.8}, \eqref{e5.9}, \eqref{e5.10} and  \eqref{e5.11} will be given in Sections~\ref{sd37}
and  ~\ref{sd18} below, which are obtained by  using    the cancellation of the kernels  of $K^{j,k} $ and $s\frac{d}{ds}K^{j,k}_s$
to show almost orthogonality properties
for these operators and use certain estimates for oscillatory integrals to establish the decay estimates.
\end{proof}

\subsubsection{Reduction to oscillatory integral operators}

To  show   estimates  \eqref{e5.8}, \eqref{e5.9}, \eqref{e5.10} and  \eqref{e5.11}, we will reduce them to  estimates
for oscillatory integral operators. Recall that
  $E_{d-1}$ is a $d\times (d-1)$ matrix  and $F_{k}$ is a   $(d-1)\times k$ matrix with $k\in [2, d)$, which are given by
\begin{equation}\label{e5.12}
 	E_{d-1}=
	\left(
	\begin{array}{cccccc}
		I_{d-1} \\
		0 \\
	\end{array}
	\right), \ \ \ \
	F_{k}=
	\left(
	\begin{array}{cccccc}
		I_{k} \\
		0 \\
	\end{array}
	\right),
\end{equation}
respectively, and we denote $E_{d-1}^{T}$, $F_{k}^{T}$ as their transposes. We  write
\begin{eqnarray*}
	f\ast K^{j,k}(x, u)=\int_{\mathbb R^{d+1}} K^{j,k}(x-y,u-v+x''^{T}E_{d-1}^{T}JE_{d-1}y'') f(y, v) dydv,
\end{eqnarray*}
where
\begin{eqnarray*}
	K^{j,k}(x-y,u-v+x''^{T}E_{d-1}^{T}JE_{d-1}y'')
	=\chi_{2}(x''-y'')\iint_{\R\times\R} e^{i(\sigma(x_{d}-y_d-\Gamma_2(x''-y'')) + \tau (u-v+x''^{T}E_{d-1}^{T}JE_{d-1}y'') )}
	\phi_{j}(\sigma)\phi_{k}(\tau)d\sigma d\tau.
	\end{eqnarray*}
We then apply a Fourier transform on ${\mathbb R^2}$, in the $(x_d, u)$ variables of   $f\ast K^{j,k}(x, u),$
and use the fact that
$
K^{j,k}(x,u)=\chi_{2}(x'')\mathcal{F}^{-1} (\phi_{j})\big(x_{d}-\Gamma_2(x'')\big)\mathcal{F}^{-1}(\phi_{k})\big(u\big)
$
to obtain
\begin{align}\label{e5.14-521}
 	&\mathcal{F}_{d}\mathcal{F}_{u}(f\ast K^{j,k})(x'', \lambda_1, \lambda_2)\nonumber\\
 	&=\phi_{j}(\lambda_1)\phi_{k}(\lambda_2) 	\int_{\mathbb R^{d-1} }e^{-i(\lambda_1\Gamma_2(x''-y'') -\lambda_2 x''^{T}E_{d-1}^{T}JE_{d-1}y'') }
	(\mathcal{F}_{d}\mathcal{F}_{u}f)(y'',\lambda_1, \lambda_2)\chi_{2}(x''-y'') dy''.
\end{align}

We have the following result.

\begin{lemma} \label{le5.2}
Define
\begin{align*}
T_{\lambda_1, \lambda_2} g(x'')=\int_{\mathbb{R}^{d-1}} e^{-i(\lambda_{1}\Gamma_2(x''-y'')-\lambda_{2}
  x''^{T}E_{d-1}^{T}JE_{d-1}y'')}\chi_{2} (x''-y'')g(y'') dy'',
\end{align*}
 Then there exists a  constant $C>0$ depending on $\supp \chi_{2}$ and $C^{N}$ norms of $\chi_{2}$
 and $\Gamma_{2}$ on $\supp \chi_{2}$ for some sufficiently large  $N$ such that for all $\theta\in[0,1]$,
\begin{align*}
\|T_{\lambda_1, \lambda_2} g\|_{L^{2}(\mathbb{R}^{d-1})}\leq C (1+|\lambda_{1}|)^{-(1-\theta)\frac{d-1}{2}}
(1+|\lambda_{2}|)^{-\theta}\|g\|_{L^{2}(\mathbb{R}^{d-1})}.
\end{align*}
 \end{lemma}

\begin{proof}
One can  rewrite
\begin{align*}
T_{\lambda_1, \lambda_2}g(x'')=\int_{\mathbb{R}^{d-1}} e^{-i\lambda_{1} \Psi_{\lambda_1, \lambda_2}(x'',y'')}\chi_{2}(x''-y'')g(y'')dy'',
\end{align*}
where the phase function $\Psi_{\lambda_1, \lambda_2}$ is given by
\begin{align}\label{e5.14}
\Psi_{\lambda_1, \lambda_2}(x'',y'')=\Gamma_2(x''-y'')-\lambda_{1}^{-1}\lambda_{2} x''^{T}E_{d-1}^{T}JE_{d-1}y''.
\end{align}
We have
$
\partial_{x'',y''}^{2}\Psi_{\lambda_1, \lambda_2}(x'',y'')=-\partial_{x''}^{2}\Gamma_2(x''-y'')-\lambda_{1}^{-1}\lambda_{2} E_{d-1}^{T}JE_{d-1}.
$
  Since $\partial_{x''}^{2}\Gamma_2(x''-y'')$ is  a positive definite matrix and $E_{d-1}^{T}JE_{d-1}$ is skew-symmetric, it follows that for all $|X|=1$,
\begin{align*}
|\partial_{x'',y''}^{2}\Psi_{\lambda_1, \lambda_2}(x'',y'')X|
&\geq |\lb \partial_{x''}^{2}\Gamma_2(x''-y'')X+\lambda_{1}^{-1}\lambda_{2} E_{d-1}^{T}JE_{d-1}X,X\rb|\nonumber\\
&=|\lb \partial_{x''}^{2}\Gamma_2(x''-y'')X,X\rb|\geq c
\end{align*}
for some $c>0$.
By \eqref{e5.14}, we have
\begin{align*}
|\partial^{\alpha}_{x'',y''}\Psi_{\lambda_1, \lambda_2}(x'',y'')|=|\partial^{\alpha}_{x'',y''}\Gamma_2(x''-y'')|\leq C_{\alpha}
\end{align*}
  for all $x''-y''\in \supp\chi_{2}$ and $|\alpha|\geq 3$.
Then by support condition of $\chi_{2}$ and Lemma \ref{le4.2},
\begin{align}\label{e5.15}
\|\chi_{2} T_{\lambda_1, \lambda_2} g\|_{L^{2}(\mathbb{R}^{d-1})}
\leq C(1+|\lambda_{1}|)^{-\frac{d-1}{2}}\|g\|_{L^{2}(\mathbb{R}^{d-1})}.
\end{align}
Observe that  $\det J_{k}\neq0$ and $x''-y''\in \supp \chi_{2}$.  This tells us that
\begin{align}\label{e5.16}
\inf\limits_{|X|=1}|J_{k}X|\geq c_{J_{k}} \textrm{\ \ \ \  and\ \ \ \ } \sup\limits_{|X|=1}|F_{k}^{T}\partial_{x''}^{2}\Gamma_2(x''-y'')F_{k}X|\leq C_{\Gamma_{2},k}
\end{align}
Take  $M= {2C_{\Gamma_{2},k}}/{c_{J_{k}}}$. Now we consider two cases: $M|\lambda_{1}|\geq |\lambda_{2}|$  and
$M|\lambda_{1}|\leq |\lambda_{2}|$.

 \medskip

 \noindent
 {\bf Case 1: $M|\lambda_{1}|\geq |\lambda_{2}|$}. In this case,  we use  \eqref{e5.15} to obtain
\begin{align}\label{e5.17}
\|\chi_{2} T_{\lambda_1, \lambda_2} g\|_{L^{2}(\mathbb{R}^{d-1})}
\leq C(1+|\lambda_{1}|)^{-(1-\theta)\frac{d-1}{2}}
(1+|\lambda_{2}|)^{-\theta}  \|g\|_{L^{2}(\mathbb{R}^{d-1})}.
\end{align}

\medskip

 \noindent
 {\bf Case 2: $M|\lambda_{1}|\leq |\lambda_{2}|$}. In this case,   we rewrite
 \begin{align*}
T_{\lambda_1, \lambda_2} g(x'')=\int_{\mathbb{R}^{d-1}} e^{-i\lambda_{2} \tilde{\Psi}_{\lambda_1, \lambda_2}(x'',y'')}\chi_{2}(x''-y'')g(y'')dy'',
\end{align*}
where $\tilde{\Psi}_{\lambda_1, \lambda_2}(x'',y'')=\lambda_{1}\lambda_{2}^{-1}\Gamma_2(x''-y'')- x''^{T}E_{d-1}^{T}JE_{d-1}y''$ and thus
\begin{align*}
\partial_{x'',y''}^{2}\tilde{\Psi}_{\lambda_1, \lambda_2}(x'',y'')=-\lambda_{1}\lambda_{2}^{-1}
\partial_{x''}^{2}\Gamma_2(x''-y'')- E_{d-1}^{T}JE_{d-1}.
\end{align*}
 By a triangle inequality and \eqref{e5.16},
\begin{align*}
|\partial_{(x_{1},...,x_{k}),(y_{1},...,y_{k})}^{2}\tilde{\Psi}_{\lambda_1, \lambda_2}(x'',y'')X|&=|-\lambda_{1}\lambda_{2}^{-1} F_{k}^{T}
\partial_{x''}^{2}\Gamma_2(x''-y'')F_{k}X-J_{k}X|\\
&\geq|J_{k}X|-M^{-1}|F_{k}^{T}\partial_{x''}^{2}\Gamma_2(x''-y'')F_{k}X|\\
&\geq c_{J_{k}}-M^{-1}c_{\Gamma_{2},k}\geq c_{J_{k}}/2
\end{align*}
for all $|X|=1$. Notice that  for all $|\alpha|\geq2$,
 $|\partial_{x'',y''}^{\alpha}\tilde{\Psi}_{\lambda_1, \lambda_2}(x'',y'')|\leq C_{\alpha}.
 $
We then apply Lemma \ref{le4.2} in the first $k$ ($k\geq2$) variables and Fubini's theorem to get
\begin{align*}
\|\chi_{2} T_{\lambda_1, \lambda_2}g\|_{L^{2}(\mathbb{R}^{d-1})}\leq C (1+|\lambda_{2}|)^{-1}\|g\|_{L^{2}(\mathbb{R}^{d-1})}.
\end{align*}
This, together with \eqref{e5.15}, tells us that for all $M|\lambda_{1}|\leq |\lambda_{2}|$
\begin{align}\label{e5.18}
\|\chi_{2} T_{\lambda_1, \lambda_2} g\|_{L^{2}(\mathbb{R}^{d-1})}\leq C(1+|\lambda_{1}|)^{-(1-\theta)\frac{d-1}{2}}
(1+|\lambda_{2}|)^{-\theta}\|g\|_{L^{2}(\mathbb{R}^{d-1})}.
\end{align}
Therefore, it follows \eqref{e5.17} and
  \eqref{e5.18} that  for all $\lambda_{1},\lambda_{2}\in \mathbb{R}$
\begin{align*}
\|\chi_{2} T_{\lambda_1, \lambda_2} g\|_{L^{2}(\mathbb{R}^{d-1})}\leq C(1+|\lambda_{1}|)^{-(1-\theta)\frac{d-1}{2}}
(1+|\lambda_{2}|)^{-\theta}\|g\|_{L^{2}(\mathbb{R}^{d-1})}.
\end{align*}
Then by a translation invariance argument similarly as in \cite[p. 236]{St2}, we have
\begin{align*}
\|T_{\lambda_1, \lambda_2} g\|_{L^{2}(\mathbb{R}^{d-1})}
\leq C(1+|\lambda_{1}|)^{-(1-\theta)\frac{d-1}{2}}
(1+|\lambda_{2}|)^{-\theta}\|g\|_{L^{2}(\mathbb{R}^{d-1})},
\end{align*}
which completes the proof of Lemma~\ref{le5.2}.
\end{proof}

\subsubsection{Proof of \eqref{e5.8} and \eqref{e5.9}}\label{sd37}

\begin{proof} [{\rm PROOF OF \eqref{e5.8}}]
By Plancherel's theorem, \eqref{e5.14-521}, Fubini's theorem and Lemma~\ref{le5.2}, we have
\begin{align*}
||f\ast K^{j,k}||_{L^{2}(\mathbb{R}^{d+1})}&=||\mathcal{F}_{d}\mathcal{F}_{u}(f\ast K^{j,k})||_{L^{2}(\mathbb{R}^{d+1})}\\
&=\big\|T_{\lambda_1, \lambda_2}\big(\mathcal{F}_{d}\mathcal{F}_{u}f(\cdot,\lambda_1, \lambda_2)\big)(x'')\big\|_{L^{2}_{\lambda_1,
\lambda_2}(\mathbb{R}^{2})L^{2}_{x''}(\mathbb{R}^{d-1})}\nonumber\\
&\leq C 2^{-(1-\theta)\frac{d-1}{2}j}2^{-\theta k}
||\mathcal{F}_{d}\mathcal{F}_{u}f(x'',\lambda_1, \lambda_2)||_{L^{2}_{\lambda_1, \lambda_2}(\mathbb{R}^{2})L^{2}_{x''}(\mathbb{R}^{d-1})}\\
&\leq C2^{-(1-\theta)\frac{d-1}{2}j}2^{-\theta k}||f||_{L^{2}(\mathbb{R}^{d+1})}
\end{align*}
for all $j,k\geq0 $ and $\theta\in[0,1]$.
\end{proof}

\begin{proof}[{\rm PROOF OF \eqref{e5.9}}]
By a scaling argument, the proof of \eqref{e5.9} reduces to show for all $n\leq0$,
\begin{align}\label{e5.20}
	\left\|f\ast  (K^{j,k}  )^{\ast} \ast K^{j,k}_{2^{n}}  \right\|_{L^{2}(\mathbb{R}^{d+1})}\leq C  2^{4j+4k}2^{n}\|f\|_{L^{2}(\mathbb{R}^{d+1})}
\end{align}
and \begin{align}\label{e5.21}
	\left\|f\ast  (K^{j,k}_{2^{n}})^{\ast} \ast K^{j,k} \right\|_{L^{2}(\mathbb{R}^{d+1})}\leq C  2^{4j+4k}2^{n}\|f\|_{L^{2}(\mathbb{R}^{d+1})}.
\end{align}

Let us prove estimate \eqref{e5.20}. We use integration by parts to see   that   for all $|\alpha|\leq 1 $ and $N>0,$
there exists a   constant $C=C(\alpha, N)>0$ such that
\begin{align}\label{e5.22}
|\partial_{x,u}^{\alpha}K^{j,k}(x,u)| + |\partial_{x,u}^{\alpha}(K^{j,k})^{\ast}(x,u)|
\leq C 2^{2j+2k}\chi_{2}(x'')\bigg(1+|x_{d}|+|u|\bigg)^{-N}.
\end{align}
Note that  $|x''^{T}E_{d-1}^{T}JE_{d-1}y''|\leq2c\|J\|\cdot|y''|\leq 2c^{2}\|J\|$ if $|x''|,|y''|\leq c$.
These, in combination the condition  $x^{T}Jy=x''^{T}E_{d-1}^{T}JE_{d-1}y''$ and the cancellation property \eqref{e5.4} of the kernel of $K^{j,k}$, yield
\begin{eqnarray*}
|(\Kone)^{\ast}\ast K^{j,k}_{2^{n}}(x,u)|
 &=&\big|\int_{\mathbb R^{d+1}} (\Kone)^{\ast}(x-y,u-v-x^{T}Jy)\Kone_{2^{n}}(y,v) dydv\big|\nonumber\\
&=&\big|\int_{\mathbb R^{d+1}} \bigg((\Kone)^{\ast}(x-y,u-v-x''^{T}E_{d-1}^{T}JE_{d-1}y'')-(\Kone)^{\ast}(x,u)\bigg) K^{j,k}_{2^{n}}(y,v) dydv\big|\nonumber\\
 &\leq&
 \int_{0}^{1}\int_{\R^{d+1}}2^{2j+2k}\bigg(1+|x-\theta y|+|u-\theta v|\bigg)^{-N}
\bigg(|y|+|v|\bigg)
 \big|K^{j,k}_{2^{n}}(y,v)\big|dydvd\theta.
\end{eqnarray*}
This gives that for $n\leq 0$,
\begin{eqnarray*}
 \int_{\mathbb R^{d+1}}|(\Kone)^{\ast}\ast K^{j,k}_{2^{n}}(x,u)|dxdu
&\leq& C 2^{2j+2k}\int_{\mathbb R^{d+1}}\big(|y|+|v|\big)\big|K^{j,k}_{2^{n}}(y,v)\big|dydv\nonumber\\
&\leq& C^{2}2^{4j+4k}2^{n}.
\end{eqnarray*}
By the Schur lemma, \eqref{e5.20} follows readily.

For \eqref{e5.21}, we use the cancellation property of the kernel of $(K^{j,k}_{2^{n}})^\ast$
and a similar argument as above to show that
$$
\int_{\mathbb R^{d+1}} |(K^{j,k}_{2^{n}})^{\ast}\ast K^{j,k}(x,u)  |dxdu\leq C 2^{4j+4k}2^{n}.
$$
This, together with the Schur lemma,  gives the desired estimate \eqref{e5.21}.
\end{proof}

\subsubsection{Proof of \eqref{e5.10} and \eqref{e5.11}}\label{sd18}

\begin{proof} [{\rm PROOF OF \eqref{e5.10}}]
  By \eqref{e5.2} and a change of variables in $\sigma,\tau$, we have
\begin{align}\label{e5.23}
K^{j,k}_{s}(x,u)
&=s^{-d+1}\chi_{2}(s^{-1}x'')\iint_{\R\times\R} e^{i\sigma(x_{d}-s\Gamma_2(s^{-1}x''))}e^{i\tau u}
\phi_{j}(s\sigma)\phi_{k}(s^{2}\tau)d\sigma d\tau.
\end{align}
Note that $s\frac{d}{ds}K^{j,k}_s(x,u)$ can be written as the finite sum of following terms,
\begin{eqnarray}\label{e5.24}
&&\hspace{-1cm} s\frac{d}{ds}K^{j,k}_s(x,u)=(K^{(1)})^{j,k}_s+ (K^{(2)})^{j,k}_s + 2^{j}(K^{(3)})^{j,k}_s
\end{eqnarray}
where
\begin{eqnarray*}
(K^{(1)})^{j,k}_s&=&\left[(-d+1)\chi_{2}(s^{-1}x'')-  s^{-1}x''\chi_{2}'(s^{-1}x'') \right]
s^{-d+1} \iint_{\R\times\R} e^{i\sigma(x_{d}-s\Gamma_2(s^{-1}x''))}e^{i\tau u}
\phi_{j}(s\sigma)\phi_{k}(s^{2}
\tau)d\sigma d\tau\\
 (K^{(2)})^{j,k}_s&=&
s^{-d+1}\chi_{2}(s^{-1}x'')\iint_{\R\times\R} e^{i\sigma(x_{d}-s\Gamma_2(s^{-1}x''))}e^{i\tau u}\left[2^{-j}s\sigma
\phi'_{j}(s\sigma)\phi_{k}(s^{2}\tau) + \phi_{j}(s\sigma)2^{-k}s\sigma\phi'_{k}(s^{2}
\tau)\right]d\sigma d\tau\\
(K^{(3)})^{j,k}_s &=&  \left[ -is^{-d+1}\Gamma_2(s^{-1}x'')  +i s^{-1}x''\Gamma_2'(s^{-1}x'') \right]
\chi_{2}(s^{-1}x'')\iint_{\R\times\R} e^{i\sigma(x_{d}-s\Gamma_2(s^{-1}x''))}e^{i\tau u}
2^{-j}s\sigma\phi_{j}(s\sigma)\phi_{k}(s^{2}
\tau)d\sigma d\tau.
\end{eqnarray*}
The method as in the proof of \eqref{e5.8} yields that  for all $i=1, 2, 3$, we have
\begin{align*}
\|f\ast (K^{i})^{j,k}\|_{L^{2}(\mathbb{R}^{d+1})}\leq C 2^{-(1-\theta)\frac{d-1}{2}j}2^{-\theta k}||f||_{L^{2}(\mathbb{R}^{d+1})}, \ \ \ \ \forall
 \theta\in[0,1].
 \end{align*}
These, together with a scaling argument, imply
\begin{align*}
\sup_{t\in[1,2]}\left\|f\ast \left[s\frac{d}{ds}K^{j,k}_s\right]_{s=2^{n}t}\right\|_{L^{2}(\mathbb{R}^{d+1})}
&\leq2^{j}\sum_{i=1}^{m}\sup_{t\in[1,2]}\big\| f\ast (K^{i})^{j,k}_{2^{n}t}\big\|_{L^{2}(\mathbb{R}^{d+1})}\nonumber\\
&\leq  C2^{j}2^{-(1-\theta)\frac{d-1}{2}j}2^{-\theta k}
\|f\|_{L^{2}(\mathbb{R}^{d+1})}\nonumber\\
&=C2^{-(\frac{d-3}{2}-\frac{d-1}{2}\theta)j}2^{-\theta k}||f||_{L^{2}(\mathbb{R}^{d+1})}.
\end{align*}
Estimate \eqref{e5.10} follows readily.
\end{proof}

\begin{proof} [{\rm PROOF OF \eqref{e5.11}}]   By \eqref{e5.24},
	we can follow the proof of \eqref{e4.16} to obtain
\begin{align*}
\left\|f\ast \left(\left[s\frac{d}{ds}K^{j,k}_s\right]_{s=2^{n}t}\right)^{\ast}
\ast \left[s\frac{d}{ds}K^{j,k}_s\right]_{s=2^{n'}t} \right\|_{L^{2}(\mathbb{R}^{d+1})}
&\leq2^{2j}\sum_{i,i'=1}^{m}\big\|f\ast \big((K^{i})^{j,k}_{2^{n}t}\big)^{\ast}\ast(K^{i'})^{j,k}_{2^{n'}t}
\big\|_{L^{2}(\mathbb{R}^{d+1})}\nonumber\\
&\leq C^{2}2^{6j+6k}2^{-|n-n'|}||f||_{L^{2}(\mathbb{R}^{d+1})}.
\end{align*}
This completes the proof of \eqref{e5.11}.
 \end{proof}

\subsection{ Weak type $(1,1)$-estimate }
For $p=1,$ we have the following result.
\begin{lemma}\label{le5.3}
Let  $j,k\geq 0$. For all $\alpha>0$, we have
\begin{align*}
\big| \big\{(x, u): \sup_{t>0}|f\ast K^{j,k}_{t}| > \alpha \big\} \big| \leq C (j+k)2^{j} \alpha^{-1} ||f||_{L^{1}(\mathbb{R}^{d+1})}.
\end{align*}
\end{lemma}

\begin{proof} The proof of Lemma~\ref{le5.3} is essential similar to that of Lemma~\ref{le4.4}.
Let  $\rho(x,u):=|x|+|u|^{1/2}$ and $c_{J}=\frac{3}{2}\|J\|^{1/2}+1$.   Like Lemma~\ref{le4.4},
by a scaling argument, it suffices to show for all $\rho(y,v)\leq 2^{-n}\delta$ and $\delta>0$,
\begin{align}\label{e5.25}
\sum_{n\in \mathbb{Z}}\int_{\rho(x,u)\geq 2c_{J}2^{-n}\delta}\sup_{s\in[1,2]}
\big|K^{j,k}_{s}\big((y,v)^{-1}(x,u)\big)-K^{j,k}_{s}(x,u)\big|dxdu\leq C (j+k)2^{j}.
\end{align}

To show \eqref{e5.25}, we first note that from \eqref{e5.2}, we have that for any $j, k\geq 0,$
\begin{align}\label{e5.26}
\big| K^{j,k}(x,u)\big|
\leq C2^{j+k}(1+|x|+|u|^{1/2})^{-N/2} \bigg(1+2^{j}|s^{-1}x_{d}-\Gamma_2(x'')|\bigg)^{-N/2}\big(1+2^{k/2}|u|^{1/2}\big)^{-N/2},
\end{align}
\begin{align}\label{7-26-e5.27}
\sup_{s\in[1,2]}\big|\nabla_{x,u}K^{j,k}_{s}(x,u)\big|
\leq C2^{2j+2k}(1+|x|+|u|^{1/2})^{-N/2}
\end{align}
for  large  $N\geq 1$; and by \eqref{e5.24},
\begin{align}\label{e5.27}
\big|\frac{d}{ds}K^{j,k}_s(x,u)\big|
\leq C2^{2j}2^{k}(1+|x|+|u|^{1/2})^{-N/2} \bigg(1+2^{j}|s^{-1}x_{d}-\Gamma_2(s^{-1}x'')|\bigg)^{-N/2}\big(1+2^{k/2}|u|^{1/2}\big)^{-N/2}.
\end{align}
 uniformly in $s\in[1,2]$.

In the case  $2^{-n}\delta\geq 1$, it follows by  \eqref{e5.26}, \eqref{e5.27} and \eqref{e4.35} that
\begin{align*}
\sum_{2^{-n}\delta\geq1}\int_{\rho(x,u)\geq 2c_{J} 2^{-n}\delta}\sup_{s\in[1,2]}
\big|K^{j,k}_{s}\big((y,v)^{-1}(x,u)\big)-K^{j,k}_{s}(x,u)
\big|dxdu\leq C2^{j}.
\end{align*}
It remains to show for all $\rho(y,v)\leq 2^{-n}\delta$,
\begin{align}\label{e5.28}
\sum_{2^{-n}\delta\leq1}\int_{\rho(x,u)\geq 2c_{J} 2^{-n}\delta}\sup_{s\in[1,2]}
\big|K^{j,k}_{s}\big((y,v)^{-1}(x,u)\big)-K^{j,k}_{s}(x,u)\big|dxdu\leq C (j+k)2^{j}.
\end{align}
By the compact $x''$-support of $K^{j,k}$ and $\rho(y,v)\leq 2^{-n}\delta\leq 1$, we can suppose $|x''|\leq C+1$ and
\begin{align}\label{e5.29}
|(y,v+y^{T}Jx)|&=|(y,v+y''^{T}E_{d-1}^{T}JE_{d-1}x'')|\leq C\rho(y,v).
\end{align}
For all $\theta\in[0,1]$, we have $$\rho(x,u)\geq 2c_{J}\rho(y,v)\geq 2c_{J}\rho(\theta y,\theta v),$$
which, together with \eqref{e4.35}, tells us that
\begin{align*}
\rho((\theta y,\theta v)^{-1}(x,u))\geq (2c_{J})^{-1}\rho(x,u).
\end{align*}
This, together with Newton-Leibniz formula, \eqref{7-26-e5.27} and \eqref{e5.29}, implies that
\begin{align*}
\sup_{s\in[1,2]}\big|K^{j,k}_{s}\big((y,v)^{-1}(x,u)\big)-K^{j,k}_{s}(x,u)
\big|\leq C 2^{2j+2k}\rho(y,v)\bigg(1+\rho(x,u)\bigg)^{-N/2},
\end{align*}
which implies that for all $\rho(y,v)\leq 2^{-n}\delta$,
\begin{align}\label{e5.30}
&\sum_{2^{-n}\delta\leq2^{-2j-2k}}\int_{\rho(x,u)\geq 2c_{J}2^{-n}\delta}\sup_{s\in[1,2]}
\big|K^{j,k}_{s}\big((y,v)^{-1}(x,u)\big)-K^{j,k}_{s}(x,u)\big|dxdu\nonumber\\
&\leq C \sum_{2^{-n}\delta\leq2^{-2j-2k}}2^{2j+2k}2^{-n}\delta\leq C.
\end{align}
On another hand, we apply  \eqref{e5.26} and \eqref{e5.27} to get
\begin{align}\label{e5.31}
&\sum_{2^{-2j-2k}\leq2^{-n}\delta\leq1}\int_{\rho(x,u)\geq 2c_{J}2^{-n}\delta}
\sup_{s\in[1,2]}\big|K^{j,k}_{s}\big((y,v)^{-1}(x,u)\big)-K^{j,k}_{s}(x,u)\big|dxdu\nonumber\\
&\leq C \sum_{2^{-2j-2k}\leq2^{-n}\delta\leq1}2^{j}\leq C (j+k)2^{j},
\end{align}
where in the last inequality we use the fact that the summation has only
$O(j+k)$ terms. Estimate \eqref{e5.28} then follows from \eqref{e5.30} and \eqref{e5.31}.
Hence, we have proved \eqref{e5.25}.
\end{proof}

\subsection{ Proof of (ii) of Theorem~\ref{th2.1} }
\begin{proof}[Proof of (ii) of Theorem~\ref{th2.1}]

First we consider $1<p\leq 2$. By \eqref{e5.3}, we have
$$\sup_{t>0}|f\ast (\mu_2)_{t}|\leq C\sum_{j=0}^{\infty} \sum_{k=0}^{\infty}\sup_{t>0}\big||f|\ast K^{j,k}_{t}\big|.$$

From \eqref{e55}, it suffices to consider that  $j+k\geq1$.
By interpolation, it follows from Lemmas  \ref{le5.1} and \ref{le5.3} that  for all $\theta\in[0,1]$,
$$
\bigg\|\sup_{t>0}\big||f|\ast (K^{j,k})_{t}\big|\bigg\|_{L^{p}(\R^{d+1})}\leq C 2^{[ (\frac{2}{p}-1)
 -(d-2)(1-\frac{1}{p})]j}  2^{ {d-1\over 2} \theta j}2^{-2 (1-\frac{1}{p})\theta k}\|f\|_{L^{p}(\R^{d+1})}.
$$
As a result,   the summation in $j,k$ converges whenever  $ d/(d-1)<p\leq 2$ with  $d\geq 3$ and $\theta$ sufficiently small.
By interpolation with a trivial $L^{\infty}$ estimate, we obtain  estimate
 (ii) of Theorem~\ref{th2.1}.
\end{proof}

\bigskip

 \noindent
{\bf Acknowledgments.}
The authors  thank Shaoming Guo   for helpful  suggestions and
discussions.
 L.X. Yan was supported by the NNSF of China, Grant
Nos.  11871480 and  12171489.


\bigskip

\bibliographystyle{amsplain}

\begin{thebibliography}{99}

 \bibitem{ACPS} T.C. Anderson, L. Cladek, M. Pramanik, and A. Seeger.
 Spherical means on the Heisenberg group: stability of a maximal function estimate. J. Anal. Math. 145 (2021), no. 1, 1--28.



 \bibitem{Bo} J. Bourgain.  Averages in the plane over convex curves and maximal operators.
 {\it J. Analyse Math.}  {\bf 47} (1986), 69--85.

 \bibitem{C} M. Cowling.
On Littlewood-Paley-Stein theory.
 {\it Rend. Circ. Mat. Palermo (2)}  {\bf 1} (1981), 21--55.


 \bibitem{Fi} V. Fischer.  The spherical maximal function on the free two-step nilpotent Lie group.
 {\it  Math. Scand.} {\bf 99} (2006),  99--118.



\bibitem{H}  L. H\"ormander. Oscillatory integrals and multipliers on $FL^{p}$. {\it Ark. Mat.} {\bf 11} (1973), 1--11.

\bibitem{MSS}  G. Mockenhaupt. A. Seeger and C.D. Sogge. Wave front sets, local
smoothing and Bourgain's circular maximal theorem. {\it Ann. of Math.} {\bf 136} (1992), 207--218.

\bibitem{MS}  D. M\"uller and A. Seeger.
 Singular spherical maximal operaotrs on a class of two step nilpotent Lie
groups. {\it Israel J. Math.} {\bf 141} (2004), 315--340.

\bibitem{NT} A. Nevo and S. Thangavelu.
 Pointwise ergodic theorems for radial averages on the Heisenberg group.
 {\it Adv. Math.}  {\it 127} (1997), 307--334.

\bibitem{NT1} E.K. Narayanan and S. Thangavelu.  An optimal theorem for the spherical maximal
operator on the Heisenberg group. {\it Israel J. Math.} {\bf  144}  (2004), 211--219.

\bibitem{PS1} M. Pramanik and A. Seeger.  $L^p$ regularity of averages over curves and bounds
for associated maximal operators. {\it Amer. J. Math.} {\bf  129}  (2007), 61--103.

\bibitem{PS2} M. Pramanik and A. Seeger.  Sobolev estimates for a class of integral operators
with folding canonical relations. Available at arXiv:1909.04173.

\bibitem{Sc}
W. Schlag. A geometric proof of the circular maximal theorem. {\it Duke Math. J.}, {\bf 93} (1998), 505--533.

 \bibitem{SS} W. Schlag and C.D. Sogge. Local smoothing estimates related to the circular maximal theorem.
 {\it Math. Res. Lett.} {\bf  4} (1997), 1--15.


\bibitem{St1}
E.M. Stein. Maximal functions. I. Spherical means. {\it Proc. Nat. Acad. Sci. U.S.A.}, {\bf 73} (1976), 2174--2175.

\bibitem{St2} E.M. Stein.  {\it Harmonic analysis: Real variable methods, orthogonality and oscillatory integrals}.
 Princeton Univ. Press, Princeton, NJ, 1993.


 \bibitem{SW} E.M. Stein and S. Wainger. Problems in harmonic analysis related to curvature.
 {\it Bull. Amer. Math. Soc.} {\bf  84} (1978), 1239--1295.


 \end{thebibliography}
 
\end{document}